\documentclass{birkjour}

\usepackage{amsmath, amssymb}
\usepackage{amsfonts}
\usepackage{graphicx}
\usepackage{mathrsfs}
\usepackage[arrow,matrix,curve,cmtip,ps]{xy}
\usepackage{amsthm}
\usepackage{enumerate}
\usepackage{color}
\usepackage[colorlinks]{hyperref}
\usepackage{tikz}

 \newtheorem{thm}{Theorem}[section]
 \newtheorem{cor}[thm]{Corollary}
 \newtheorem{lem}[thm]{Lemma}
 \newtheorem{prop}[thm]{Proposition}
 \theoremstyle{definition}
 \newtheorem{defn}[thm]{Definition}
 \theoremstyle{remark}
 \newtheorem{rem}[thm]{Remark}
 \newtheorem*{ex}{Example}
 \numberwithin{equation}{section}

\newcommand{\Z}{\mathbb{Z}}
\newcommand{\N}{\mathbb{N}}

\newcommand{\C}{\mathbb{C}}
\newcommand{\g}{\mathcal{G}_{G,\Lambda}}
\newcommand{\La}{\Lambda}
\newcommand{\Lam}{\Lambda^\mathrm{min}}
\newcommand{\la}{\lambda}
\newcommand{\og}{\mathcal{O}_{G,\Lambda}}

\newcommand{\es}{\mathcal{E}(\mathcal{S}_{G,\Lambda})}
\newcommand{\sg}{\mathcal{S}_{G,\Lambda}}

\newcommand{\gt}{\mathcal{G}_{\mathrm{tight}}(\mathcal{S}_{G,\Lambda})}

\newcommand{\ca}{\mathfrak{C}}

\begin{document}

%
%
%
%
%
%
%
%
%

\title[self-similar $k$-graph $C^*$-algebras]{An inverse semigroup approach to self-similar k-graph $C^*$-algebras and simplicity}

\author[Hossein Larki]{Hossein Larki}

\address{Department of Mathematics\\
Faculty of Mathematical Sciences and Computer\\
Shahid Chamran University of Ahvaz\\
P.O. Box: 83151-61357\\
Ahvaz\\
 Iran}

\email{h.larki@scu.ac.ir}



\date{\today}


\begin{abstract}
We generalize the Li-Yang notion of self-similar $k$-graph $(G,\La)$ and its $C^*$-algebra $\mathcal{O}_{G,\La}$ to any finitely aligned  $k$-graph $\La$. We then introduce an inverse semigroup model for $\mathcal{O}_{G,\La}$ and analyze its tight groupoid and $C^*$-algebra via inverse semigroup methods.
\end{abstract}

\maketitle


\section{Introduction}

The class of self-similar graph $C^*$-algebras is introduced by R. Exel and E. Pardo \cite{exe17} as a unified framework for the Katsura's \cite{kat08} and Nakrashevyche's algebras \cite{nek04,nek05} besides the graph $C^*$-algebras \cite{rae05}, which is interested in recent years \cite{exe17,lar21, cla19, bed17,li21-ideal,haz19,lar23, lac18}. In particular,  a specific inverse semigroup and tight groupoid model is introduced to this class of $C^*$-algebras, and some $C^*$-algebraic features are characterized via the results of \cite{exe16}.

Let $G$ be a (discrete and countable) group and $\La=(\La^0,\La,r,s)$ be a row-finite source-free $k$-graph. Inspired from \cite{exe17}, Li and Yang defined a notion of self-similar action $G \curvearrowright \La$ and associated universal $C^*$-algebra $\og$ \cite{li21,li21-ideal}, which are naturally the higher-rank analogues of the Exel-Pardo ones and contain moreover the boundary quotient $C^*$-algebras $\mathcal{Q}(\La\bowtie G)$ due to \cite[Theorem 3.3]{li19-boundary} (cf. \cite[Example 3.10(4)]{li21}).  To study the $C^*$-algebraic properties, they also introduce a path-like groupoid $\mathcal{G}_{G,\La}$ such that $\og\cong C^*(\mathcal{G}_{G,\La})$ provided $G$ is an amenable group and $(G,\La)$ is pseudo free \cite[Theorem 5.10]{li21}. Recently the algebraic analogues of Exel-Pardo $C^*$-algebras are considered in \cite{haz19,lar23}.

In this paper, we consider self-similar $k$-graphs $(G,\La)$ over a finitely aligned $k$-graph $\La$ (possibly containing sources), and define a universal $C^*$-algebra $\og$ for any such $(G,\La)$. Our main aim here is to introduce an inverse semigroup $\sg$ and its tight groupoid of germs $\gt$ such that
$$\og \cong C^*_{\mathrm{tight}}(\sg) \cong C^*(\gt)$$
(Theorem \ref{thm4.7} below). Thus, by the inverse semigroup approach of \cite{exe16}, we can explicitly describe some properties of the groupoid $\gt$ such as Hausdorffness, minimality and effectiveness by the underlying graphical features. Our motivation is threefold:
\begin{enumerate}[(1)]
  \item First, in \cite{li21,li21-ideal} it was just considered row-finite $k$-graphs without sources, while we can generalize the notion to a significant larger class containing self-similar $k$-graphs over finitely aligned $\La$ (with possible sources).
  \item Next, all results of \cite{li21,li21-ideal} are for \emph{pseudo free} self-similar $k$-graphs $(G,\La)$. We will modify them for not necessarily pseudo free ones. For example, the groupoid $\mathcal{G}_{G,\La}$ of \cite[Section 5.1]{li21} is not a suitable groupoid for general row-finite $(G,\La)$ and the isomorphism $\og\cong C^*(\mathcal{G}_{G,\La})$ in \cite[Theorem 5.10]{li21} does not probably hold anymore for non-pseudo free $(G,\La)$. Using the tight groupoid of germs $\gt$ constructed from $\sg$, Theorem \ref{thm4.7} may be considered as an improvement of this result for the non-pseudo free cases. We obtain in particular Corollaries \ref{cor10.2} and \ref{cor11.4} for the simplicity of $\og$, which extend \cite[Theorem 6.6(ii)]{li21}.
  \item Finally, the inverse semigroup approach provides very useful and strong methods in its own right for the study of $C^*$-algebras (see \cite{exe08,mil14,lal16,lal19,mes16} for example), so it does as well for the self-similar $k$-graph $C^*$-algebras $\og$ even in the row-finite and pseudo free case \cite{li21,li21-ideal,li19-kms}.
\end{enumerate}

After submitting the paper to journal, we were informed by the anonymous referee that there is a more general class of $C^*$-algebras associated to \emph{left cancelative small categories} (LCSCs) introduced by J. Spielberg \cite{spi14,spi20}, and some of our results in this paper may be the description of those obtained in the LCSC setting. Indeed, for any self-similar $k$-graph $(G, \La, \varphi)$ one may construct a so-called \emph{Zappa-Sz$\acute{e}$p product} $\La \rtimes^\varphi G$ as an LCSC such that their associated $C^*$-algebras coincide. There are several inverse semigroups and groupoids having been proposed in the literature for LCSC $C^*$-algebras, and our approach here is more closer to the inverse semigroup model of \cite{ort20}. In particular, we prove in Proposition \ref{prop6.5} that the inverse semigroup $\sg$ of Section \ref{sec3} and $\mathcal{S}(\La \rtimes ^{\varphi} G)$ of \cite{ort20} produce isomorphic tight groupoids, and hence by combining it with \cite[Proposition 5.2]{ort20}, $\gt$ is isomorphic to the Spielberg's groupoid $G|_{\partial(\La \rtimes ^{\varphi} G)}$ \cite{spi20}. However, we should note that the arguments and proofs in the present paper are completely based on $k$-graphical theory, so they are independent of known results in the LSCS setting and much more understandable.

This paper is organized as follows. We begin in Section \ref{sec2} by recalling some basic definitions and results in the literature which are needed throughout the paper. Also, self-similar $k$-graphs $(G,\La)$ over a finitely aligned $k$-graph $\La$ and associated $C^*$-algebras $\og$ are defined. In Section \ref{sec3}, for a given finitely aligned $(G,\La)$ we define a specific inverse semigroup $\sg$ and then, in Section \ref{sec4}, show that its tight $C^*$-algebra is isomorphic to $\og$. Note that the construction of $\sg$ is more complicated than that of the inverse semigroup $\mathcal{S}_{G,E}$ of \cite{exe17} even in the 1-graph setting.

Sections \ref{sec5}, \ref{sec6}, and \ref{sec7} will be devoted to study the groupoid $\gt$. In Section \ref{sec5}, Hausdorff property of $\gt$ is characterized by giving a necessary and sufficient condition for the underlying $(G,\La)$. In particular, if $(G,\La)$ is pseudo free then $\gt$ is Hausdorff. In Section \ref{sec6}, we describe the tight spectrum of $\sg$ via the boundary path space $\partial \La$. Indeed, using a one-to-one correspondence between boundary paths of $\La$ and ultrafilters in $\es$ we prove that the ultrafilter space $\widehat{\mathcal{E}}_\infty(\sg)$ is homeomorphic to $\partial \La$. Note that, according to Lemma \ref{lem3.4}, the idempotent set $\es$ of $\sg$ coincides with that of $\mathcal{S}_\La$ in \cite{far05}. In light of this fact and that $\partial \La$ is a closed subset of the path space of $\La$ (see \cite[Section 5]{far05} for details), we conclude that the tight spectrum $\widehat{\mathcal{E}}_{\mathrm{tight}}(\sg)$ is homeomorphic to $\partial \La$. We also prove that $\gt$ is isomorphic to $\mathcal{G}_{\mathrm{tight}}(\mathcal{S}(\La \rtimes^{\varphi}G))$ of \cite{ort20} and the Spielberg's groupoid $G|_{\partial (\La \rtimes^{\varphi} G)}$ of \cite{spi14,spi20}. Next, in Section \ref{sec7}, we investigate the relation between the ``path-like" groupoid $\mathcal{G}_{G,\La}$ of \cite{li21} and the tight groupoid $\gt$ in the case that $(G,\La)$ is a pseudo free self-similar $k$-graph over a row-finite source-free $k$-graph $\La$, and prove in Proposition \ref{prop7.2} that they are isomorphic as topological groupoids.

In Sections \ref{sec8} and \ref{sec9}, we give necessary and sufficient graphical conditions for a self-similar $k$-graph $(G,\La)$, under which the groupoid $\gt$ is minimal and effective (or equivalently, the natural action $\sg \curvearrowright \widehat{\mathcal{E}}_{\mathrm{tight}}(\sg)$ is reducible and topologically free, respectively). These results will be applied, in particular, in next sections to study the simplicity of $\og$.

Finally, the last two sections will be devoted to the simplicity of $\og$. In order to apply \cite[Proposition 3.10]{li21-ideal} for the nuclearity of $\og$, we restrict our self-similar $k$-graphs $(G,\La)$ over row-finite source-free $k$-graphs $\La$. (It is an open problem in the non-row-finite case that under \emph{what conditions} the $C^*$-algebra $\og$ would be nuclear, even for a self-similar 1-graph $\La$.) In Section \ref{sec10}, we assume the groupoids $\gt$ is Hausdorff (see Theorem \ref{thm5.2} below) and generalize the simplicity result \cite[Theorem 6.6(ii)]{li21} to this case. We should also note that the arguments of \cite{li21} and \cite{li21-ideal} are based on the groupoid methods, while ours in this paper are on the inverse semigroup results of \cite{exe16}. Furthermore, we know that the simplicity of the reduced $C^*$-algebra $C^*_{\mathrm{red}}(\mathcal{G})$ for a non-Hausdorff groupoid $\mathcal{G}$ is more complicated; nevertheless, we may consider the result of \cite[Section 5]{cla19} for the tight groupoid $\gt$. Given an inverse semigroup $S$, the authors of \cite{cla19} introduced Condition (S) (which is stronger than the topologically free propoerty) for the action $S \curvearrowright \widehat{\mathcal{E}}_{\mathrm{tight}}(S)$ such that it together with the minimality of $\mathcal{G}_{\mathrm{tight}}(S)$ imply that $C^*_{\mathrm{red}}(\mathcal{G}_{\mathrm{tight}}(S))$ is simple (combine \cite[Lemma 5.6 and Corollary 4.12(2)]{cla19}). In Section \ref{sec11}, we provide certain conditions for a self-similar graph $(G,\La)$ to insure that the action $\theta :\sg \curvearrowright \widehat{\mathcal{E}}_{\mathrm{tight}}(\sg)$ satisfies Condition (S), and we then obtain a result for the simplicity of $\og$ in the non-Hausdorff case.


\section{Preliminaries}\label{sec2}

\subsection{Higher-rank graphs}

Let $\N=\{0,1,2,\ldots\}$. Fixing an integer $k\geq 1$, we consider $\N^k$ as a semigroup under pointwise addition with the generators $e_1,\ldots,e_n$ and the identity $0:=(0,\ldots,0)\in \N^k$. For any $n\in \N^k$, we write $n=(n_1,\ldots,n_k)$ and we may define the partial order $\leq$ on $\N^k$ given by $m\leq n$ if $m_i\leq n_i$ for all $1\leq i \leq k$. We also write $m\vee n$ and $m\wedge n$ for the coordinate-wise maximum and minimum, respectively.

Following \cite{kum00}, a \emph{$k$-graph} $\La=(\La^0,\La,r,s)$ is a countable small category $\La$ equipped with a \emph{degree functor} $d:\La \rightarrow \N^k$ satisfying the \emph{unique factorisation property}: if $\la\in\La$ and $d(\la)=m+n$ for $m,n\in\N^k$, then there exist unique $\mu,\nu\in \La$ such that $d(\mu)=m$, $d(\nu)=n$ and $\la=\mu \nu$. We usually denote $\la(0,m):=\mu$ and $\la(m,d(\la)):=\nu$. Using the 1-skeleton of $k$-graphs, the elements $v\in \La^0$ are usually considered as vertices and $\la\in\La^n:=d^{-1}(n)$ as a path (of rank $k$) from $s(\la)$ to $r(\la)$ with degree $n\in \N^k$.

For each $\mu\in \La$, let $\mu\La:=\{\mu\nu: \nu\in \La, r(\nu)=s(\mu)\}\subseteq \La$ and $\mu\La^n:=\mu\La \cap \La^n.$ A \emph{source} is a vertex $v\in \La^0$ such that $|v\La^{e_i}|=0$ for some $1\leq i\leq k$. Also, $\La$ is called to be \emph{row-finite} whenever $|v\La^{e_i}|<\infty$ for all $v\in \La^0$ and $1\leq i\leq k$. Given $\mu,\nu\in \La$, a \emph{minimal common extension} for $\mu$ and $\nu$ is a path $\la\in \mu\La \cap \nu\La$ such that $d(\la)=d(\mu)\vee d(\nu)$. We denote by $\mathrm{MCE}(\mu,\nu)$ the set of all minimal common extensions of $\mu$ and $\nu$. We also denote
$$\La^{\mathrm{min}}(\mu,\nu):=\left\{(\alpha,\beta)\in\La\times \La: \mu\alpha=\nu\beta\in \mathrm{MCE}(\mu,\nu)\right\}$$
and if $\mu\in\La$, $X\subseteq \La$, then
$$\mathrm{Ext}(\mu;X):=\bigcup_{\nu\in X}\left\{\alpha:(\alpha,\beta)\in \La^{\mathrm{min}}(\mu,\nu) \mathrm{~for ~some~} \beta\in\La \right\}.$$

\begin{defn}[{\cite[Definition 2.2]{rae04}}]
A $k$-graph $\La$ is called \emph{finitely aligned} if $\La^{\mathrm{min}}(\mu,\nu)$ is finite (possibly empty) for all $\mu,\nu\in\La$.
\end{defn}

Let $\La$ be a finitely aligned $k$-graph and $v\in \La^0$. A subset $X\subseteq v\La \setminus \{v\}$ is called \emph{exhaustive} if for each $\mu\in v\La$, there exists $\nu\in X$ such that $\La^{\mathrm{min}}(\mu,\nu)\neq \emptyset$. Let us denote by $\mathbf{FE}(\La)$ the collection of all finite and exhaustive sets in $\La$, that is
$$\mathbf{FE}(\La):=\bigcup_{v\in\La^0}\left\{X\subseteq v\La\setminus \{v\}: X \mathrm{~is~finite ~ and ~ exhaustive} \right\}.$$

\begin{defn}[{\cite{rae04}}]\label{defn2.2}
Let $\La$ be a finitely aligned $k$-graph. A \emph{Cuntz-Krieger $\La$-family} (or briefly \emph{$\La$-family}) is a collection $\{s_\la: \la\in \La\}$ of partial isometries in a $C^*$-algebra satisfying the following relations:
\begin{enumerate}[(CK1)]
  \item $s_v s_w=\delta_{v,w} s_v$ for all $v,w\in\La^0$,
  \item $s_\mu s_\nu=s_{\mu \nu}$ for all $\mu,\nu\in\La$ with $s(\mu)=r(\nu)$,
  \item $s_{\mu}^* s_\nu=\sum_{(\alpha,\beta)\in \La^{\mathrm{min}}(\mu,\nu)} s_\alpha s_{\beta}^*$ for all $\mu,\nu\in \La$, and
  \item $\prod_{\mu\in X}\left(s_v-s_\mu s_{\mu}^* \right)=0$ for all $v\in \La^0$ and $X\in v \mathbf{FE}(\La)$.
\end{enumerate}
\end{defn}

\subsection{Self-similar $k$-graphs and their $C^\ast$-algebras}
Let $\La$ be a finitely aligned $k$-graph. An \emph{automorphism of $\La$} is a bijection $\psi:\La\rightarrow \La$ such that $\psi(\La^n)\subseteq \La^n$ for every $n\in \N^k$ with $s\circ \psi=\psi \circ s$ and $r\circ \psi=\psi \circ r$. We write $\mathrm{Aut}(\La)$ for the group of automorphisms on $\La$. Also, if $G$ is a countable discrete group, an \emph{action} of $G$ on $\La$ is a group homomorphism $g\mapsto \psi_g$ from $G$ into $\mathrm{Aut}(\La)$.

\begin{defn}\label{defn2.3}
Let $\La$ be a finitely aligned $k$-graph and $G$ a discrete group with identity $e_G$. A triple $(G,\La,\varphi)$ is called a \emph{self-similar $k$-graph} whenever the following properties hold:
\begin{enumerate}
  \item $G$ acts on $\La$ by a group homomorphism $g\mapsto \psi_g$. We will write $g\cdot \mu$ for $\psi_g(\mu)$ to ease the notation.
  \item $\varphi:G\times \La\rightarrow G$ is a 1-cocycle for the above action $G$ on $\La$ such that for every $g\in G$, $\mu,\nu\in \La$ and $v\in \La^0$ we have
  \begin{itemize}
    \item[(a)] $\varphi(gh,\mu)=\varphi(g,h \cdot \mu)\varphi(h,\mu)  \hspace{5mm}$ (the 1-cocycle property),
    \item[(b)] $g\cdot (\mu\nu)=(g\cdot \mu)(\varphi(g,\mu) \cdot \nu) \hspace{5mm}$ (the self-similar equation),
    \item[(c)] $\varphi(g,\mu\nu)=\varphi(\varphi(g,\mu),\nu)$, and
    \item[(d)] $\varphi(g,v)=g$.
  \end{itemize}
\end{enumerate}
\end{defn}

If there is no any ambiguity, we often write a self-similar $k$-graph $(G,\La,\varphi)$ as $(G,\La)$. It is worth noting that for every $g\in G$ and $\mu\in\La$ we have
\begin{equation}\label{eq2.1}
\varphi(g^{-1},\mu)=\varphi(g,g^{-1} \cdot \mu)^{-1}
\end{equation}
by \cite[Lemma 3.5(iii)]{li21}.

{\bf Standing assumption.} Throughout the paper, we work with self-similar $k$-graphs $(G,\La)$ over finitely aligned $k$-graphs $\La$.

\begin{defn}\label{defn2.4}
Let $(G,\La)$ be a self-similar $k$-graph. A \emph{$(G,\La)$-family} is a set
$$\{s_\mu:\mu\in \La\}\cup\{u_{v,g}:v\in\La^0,g\in G\}$$
in a $C^*$-algebra satisfying the following relations:
\begin{enumerate}[(1)]
  \item $\{s_\mu:\mu\in\La\}$ is a Cuntz-Krieger $\La$-family,
  \item $u_{v,e_G}=s_v$ for all $v\in\La^0$,
  \item $u_{v,g}^*=u_{g^{-1} \cdot v ,g^{-1}}$ for all $v\in\La^0$ and $g\in G$,
  \item $u_{v,g}s_\mu=\delta_{v,g\cdot r(\mu)}s_{g\cdot \mu}u_{g\cdot s(\mu),\varphi(g,\mu)}$ for all $v\in\La^0$, $\mu\in\La$, and $g\in G$,
  \item $u_{v,g}u_{w,h}=\delta_{v,g\cdot w}u_{v,gh}$ for all $v,w\in\La^0$ and $g,h\in G$.
\end{enumerate}
The $C^*$-algebra $\og$ associated to $(G,\La)$ is the universal $C^*$-algebra generated by a $(G,\La)$-family $\{s_\mu,u_{v,g}\}$.
\end{defn}

\begin{rem}
Whenever $(G,\La)$ is a self-similar $k$-graph with a row-finite, source-free $k$-graph $\La$, we will explain in Remark \ref{rem4.8} below that our definition for $\og$ in Definition \ref{defn2.4} coincides with that of \cite[Definition 3.9]{li21-ideal}.
\end{rem}

\begin{prop}\label{prop2.6}
Let $(G,\La)$ be a self-similar $k$-graph. Then
\begin{equation}\label{equ1.1}
\og=\overline{\mathrm{span}}\left\{s_\mu u_{s(\mu),g} s_\nu^*: g\in G,~ \mu,\nu\in \La,~ \mathrm{and}~ s(\mu)=g\cdot s(\nu)\right\}.
\end{equation}
\end{prop}

\begin{proof}
If define $M:=\mathrm{span}_{\C}\{s_\mu u_{s(\mu),g} s_\nu^*: g\in G,~ \mu,\nu\in \La\}$, then proof of \cite[Proposition 2.7]{lar23} implies that $M^*=M$ and $M^2\subseteq M$. Hence $\overline{M}$ is a $C^*$-subalgebra of $\og$. Moreover, by \cite[Lemma 2.6]{lar23}, for nonzero elements $s_\mu u_{s(\mu),g} s_\nu^*$ in $M$ we must have $s(\mu)=g\cdot s(\nu)$. Now the fact that $M$ contains the generators of $\og$ concludes the identification (\ref{equ1.1}).
\end{proof}

\subsection{Groupoid $C^*$-algebras}

In this paper, we work with a specific kind of topological groupoids, namely \emph{tight groupoids of germs}, which will be defined in the Subsection \ref{sec2.4}. Let us give here a brief introduction to topological groupoids and the associated $C^*$-algebras; for more details see \cite{ren80,ana00}.

A {\it groupoid} is a small category $\mathcal{G}$ with an inverse map $\alpha\mapsto \alpha^{-1}$. One may define the range map $r(\alpha)=\alpha \alpha^{-1}$ and the source map $s(\alpha)=\alpha^{-1}\alpha$ for all $\alpha\in \mathcal{G}$, which satisfy $r(\alpha)\alpha=\alpha=\alpha s(\alpha)$. So, for every morphisms $\alpha,\beta\in \mathcal{G}$, the composition $\alpha \beta$ is well-defined if and only if $s(\alpha)=r(\beta)$. Also, the set of identity morphisms is called {\it the unit space of $\mathcal{G}$}, that is $\mathcal{G}^{(0)}:=\{\alpha^{-1}\alpha:\alpha\in \mathcal{G}\}$.

We will only deal with {\it topological groupoids} $\mathcal{G}$ such that multiplication, inverse, the range and source maps are continuous. In this case, $\mathcal{G}$ is called an {\it \'{e}tale groupoid} if its topology is locally compact and $r,s$ are local homeomorphisms. Then a {\it bisection} is a subset $B\subseteq \mathcal{G}$ such that the restricted maps $r|_B$ and $s|_B$ are both homeomorphisms. We say that $\mathcal{G}$ is {\it ample} in case the topology on $\mathcal{G}$ is generated by a basis of compact open bisections.

Since all tight groupoids of inverse semigroups in this paper are ample, we recall now the full and reduced $C^*$-algebras only for ample groupoids. Fixing an ample groupoid $\mathcal{G}$, write $C_c(\mathcal{G})$ for the $*$-algebra of compactly supported continuous functions $f:\mathcal{G} \rightarrow \mathbb{C}$ with the convolution multiplication and the involution $f^*(\alpha):=\overline{f(\alpha^{-1})}$. Given any unit $u\in \mathcal{G}^{(0)}$, define the left regular $*$-representation $\pi_u:C_c(\mathcal{G})\rightarrow B(\ell^2(\mathcal{G}_u))$ by
$$\pi_u(f)\delta_\alpha:=\sum_{s(\beta)=r(\alpha)}f(\beta) \delta_{\beta \alpha} \hspace{5mm} (f\in C_c(\mathcal{G}),~\alpha\in \mathcal{G}_u),$$
where $\mathcal{G}_u:=s^{-1}(\{u\})$. Then the {\it reduced $C^*$-algebra $C^*_r(\mathcal{G})$} is the completion of $C_c(\mathcal{G})$ under the reduced $C^*$-norm
$$\|f\|_r:=\sup_{u\in \mathcal{G}^{(0)}}\|\pi_u(f)\|.$$
Moreover, there is a {\it full $C^*$-algebra $C^*(\mathcal{G})$} associated to $\mathcal{G}$, which is the completion of $C_c(\mathcal{G})$ taken over all $\|.\|_{C_c(\mathcal{G})}$-decreasing representations of $\mathcal{G}$. Hence, $C^*_r(\mathcal{G})$ is a quotient of $C^*(\mathcal{G})$, and \cite[Proposition 6.1.8]{ana00} says that they are equal whenever $\mathcal{G}$ is amenable.

\subsection{Inverse semigroups and associated germ groupoids}\label{sec2.4}

An \emph{inverse semigroup} is a semigroup $S$ such that for each $s\in S$, there is a unique $s^*\in S$ satisfying
$$s=s s^* s \hspace{5mm} \mathrm{and} \hspace{5mm} s^*=s^*s s^*.$$
In particular, we have $e^*=e$ for all idempotents $e=e^2$ in $S$ because $e=eee$. Let us denote by $\mathcal{E}(S)$ the set of all idempotents in $S$, which is a commutative semilattice by defining $e \wedge f:=ef$. Given $e,f\in \mathcal{E}(S)$, we say that $e$ intersects $f$, denoted by $e\Cap f$, when $ef\neq 0$. There is a natural partial order on $S$ by $s\leq t$ if $s=te$ for some $e\in \mathcal{E}(S)$; in case $s\in \mathcal{E}(S)$, we have $s\leq t ~ \Longleftrightarrow ~ st=s$. According to this partial order, for any $s\in S$, we will use the notation
$$s\uparrow:=\{t\in S:t\geq s\} ~~ \mathrm{and} ~~ s\downarrow:=\{t\in S:t \leq s\}.$$

\begin{defn}
Let $S$ be an inverse semigroup containing a zero element $0$ (in the sense $0s=s0=0$ for every $s\in S$). A {\it filter} in $\mathcal{E}(S)$ is a nonempty set $\mathcal{F}\subseteq \mathcal{E}(S)$ that is closed under multiplication and satisfies $s \in \mathcal{F} ~~ \Rightarrow s\uparrow \subseteq \mathcal{F}$. A proper maximal filter in $\mathcal{E}(S)$ is called an {\it ultrafilter}. The set of all filters (ultrafilters) without 0 in $\mathcal{E}(S)$ is denoted by $\widehat{\mathcal{E}}_0(S)$ ($\widehat{\mathcal{E}}_\infty(S)$ respectively).
\end{defn}

Note that we may equip $\widehat{\mathcal{E}}_0(S)$ with a locally compact and Hausdorff topology generated by neighborhoods of the form
$$N(e;e_1,\ldots ,e_n):=\{\mathcal{F}\in \widehat{\mathcal{E}}_0(S): e\in \mathcal{F}, e_1,\ldots,e_n\notin \mathcal{F}\} \hspace{5mm} (e,e_i\in \mathcal{E}(S)).$$

\begin{defn}[\cite{exe08}]
The {\it tight spectrum of $S$} (or {\it tight filter space}), denoted by $\widehat{\mathcal{E}}_{\mathrm{tight}}(S)$, is the closure of $\widehat{\mathcal{E}}_\infty(S)$ in $\widehat{\mathcal{E}_0}(S)$. Then each filter in $\widehat{\mathcal{E}}_{\mathrm{tight}}(S)$ is called a {\it tight filter}.
\end{defn}

We usually consider $\widehat{\mathcal{E}}_{\mathrm{tight}}(S)$ as a topological space equipped with the restricted topology from $\widehat{\mathcal{E}_0}(S)$.

According to \cite{exe08}, there is a natural action of $S$ on $\widehat{\mathcal{E}}_{\mathrm{tight}}(S)$, and hence one may construct its groupiod of germs. Briefly, given any $e\in \mathcal{E}(S)$, let
$$D^e:=\{\mathcal{F}\in \widehat{\mathcal{E}}_{\mathrm{tight}}(S): e\in \mathcal{F}\}.$$
Then the action $\theta: S\curvearrowright \widehat{\mathcal{E}}_{\mathrm{tight}}(S)$, $s\mapsto \theta_s$, is defined in terms of the maps $\theta_s:D^{s^* s}\rightarrow D^{s s^*}$ given by $\theta_s(\mathcal{F})=(s\mathcal{F}s^*)\uparrow$. Now consider the set
$$S*\widehat{\mathcal{E}}_{\mathrm{tight}}(S)=\{(s,\mathcal{F})\in S\times \widehat{\mathcal{E}}_{\mathrm{tight}}(S): s^*s\in \mathcal{F}\}.$$
For two elements $(s,\mathcal{F}),(t,\mathcal{F}')\in S*\widehat{\mathcal{E}}_{\mathrm{tight}}(S)$, define $(s,\mathcal{F})\sim (t,\mathcal{F}')$ whenever $\mathcal{F}=\mathcal{F}'$ and $se=te$ for some $e\in \mathcal{F}$. Write $[s,\mathcal{F}]$ for the equivalent class of $(s,\mathcal{F})$. Then the {\it tight groupoid associated to $S$} is the groupoid $\mathcal{G}_{\mathrm{tight}}(S)=S*\widehat{\mathcal{E}}_{\mathrm{tight}}(S)/\sim$ with the multiplication
$$[t,\theta_s(\mathcal{F})][s,\mathcal{F}]:=[ts,\mathcal{F}]$$
and the inverse map
$$[s,\mathcal{F}]^{-1}:=[s^*,\theta_s(\mathcal{F})].$$
Note that the range and source maps are
$$r([s,\mathcal{F}])=[ss^*,\theta_s(\mathcal{F})] \hspace{5mm} \mathrm{and} \hspace{5mm} s([s,\mathcal{F}])=[s^*s,\mathcal{F}]$$
and we may identify the unit space of $\mathcal{G}_{\mathrm{tight}}(S)$ with $\widehat{\mathcal{E}}_{\mathrm{tight}}(S)$ by $[e,\mathcal{F}]\mapsto \mathcal{F}$.

Moreover, there is a natural locally compact topology on $\widehat{\mathcal{E}}_{\mathrm{tight}}(S)$ generated by the sets of the form
$$\Theta(s,U)=\{[s,\mathcal{F}]\in \mathcal{G}_{\mathrm{tight}}(S):\mathcal{F}\in U\}$$
where $s\in S$ and $U\subseteq \widehat{\mathcal{E}}_{\mathrm{tight}}(S)$ is open, which makes $\mathcal{G}_{\mathrm{tight}}(S)$ as an \'{e}tale groupoid.

\subsection{Left cancelative small categories}\label{sec2.5}

In this subsection, we briefly review the notion of left cancelative small categories introduced in \cite{spi12,spi14,spi20}, the \emph{Zappa-Sz$\acute{e}$p product} $\La \rtimes^\varphi G$ of a self-similar $k$-graph $(G,\La,\varphi)$ \cite{bed18}, and the associated inverse semigroups $\mathcal{S}(\La \rtimes^\varphi G)$, $\mathcal{T}(\La \rtimes^\varphi G)$ \cite{ort20}.

Given a small category $\ca$, we denote by $\ca^0$ the set of objects in $\ca$ which may be identified with the identity morphisms. So, $\ca^0\subseteq \ca$. If $r(\alpha)$ and $s(\alpha)$ are the range and source of each $\alpha$ in $\ca$ respectively, then product $\alpha \beta$ in $\ca$ is defined if and only if $s(\alpha)=r(\beta)$. We say that $\ca$ is \emph{left cancelative} whenever for every $\alpha,\beta,\gamma\in \ca$ with $s(\alpha)=r(\beta)=r(\gamma)$, $\alpha \beta=\alpha \gamma$ implies $\beta=\gamma$. \emph{Right cancelation} is defined in a similar way. The expression ``left cancelative small category" is usually abbreviated by ``LCSC" in the literature.

\begin{defn}
An LCSC $\ca$ is called \emph{finitely aligned} if for every $\alpha,\beta\in \ca$ there exists a finite subset $\Gamma \subseteq \ca$ such that $\alpha \ca \cap \beta \ca =\cup_{\gamma \in \Gamma} \gamma \ca$.
\end{defn}

Let $(G,\La,\varphi)$ be a self-similar $k$-graph as in Definition \ref{defn2.3}. Following \cite[Proposition 4.6]{bed18}, one may associate the Zappa-Sz$\acute{\mathrm{e}}$p product $\La \rtimes^\varphi G:=\La \times G$ as a small category where $(\La \rtimes^\varphi G)^0=\La^0 \times \{e_G\}$ and the range and source of $(\mu,g)\in \La \rtimes^{\varphi} G$ are defined by
$$r(\mu,g):=(r(\alpha),e_G) \hspace{5mm} \mathrm{and} \hspace{5mm} s(\mu,g):=(g^{-1}\cdot s(\mu), e_G).$$
Moreover, \cite[Propositions 4.6 and 4.13]{bed18} imply that $\La \rtimes^\varphi G$ is left cancelative and finitely aligned as $\La$ is (the left cancelation of $\La$ follows from the unique factorization property). However, $\La \rtimes^\varphi G$ is not necessarily right cancelative (see the last paragraph of Section \ref{sec5}).

There are several inverse semigroups associated to an LCSC \cite{bed18,ort20,li23} which may be considered in particular for the Zappa-Sz$\acute{\mathrm{e}}$p products $\La \rtimes^\varphi G$. Here, we recall the ones $\mathcal{S}(\ca)$ and $\mathcal{T}(\ca)$ from \cite{ort20}. For a set $X$, let $\mathfrak{I}(X)$ be the symmetric inverse semigroup of partial bijections on $X$, that is
$$\mathfrak{I}(X)=\{f:Y\longrightarrow Z: Y,Z\subseteq X ~~ \mathrm{and} ~~ f ~~ \mathrm{is ~ a ~ bijection} \}$$
endowed with the multiplication
$$gf:= g \circ f: f^{-1}(r(f) \cap s(g)) \longrightarrow g(r(f) \cap s(g))$$
and the involution $f^*:=f^{-1}$.

\begin{defn}
Let $\ca$ be an LCSC. For any $\alpha\in \ca$, we define two elements $\sigma^\alpha$ and $\tau^\alpha$ in $\mathfrak{I}(\ca)$ by
\begin{enumerate}[(1)]
  \item $\sigma^\alpha: \alpha \ca \rightarrow s(\alpha)\ca$, $\alpha \beta\mapsto \beta$, and
  \item $\tau^\alpha:s(\alpha)\ca \rightarrow \alpha \ca$, $\beta\mapsto \alpha \beta.$
\end{enumerate}
Then the left cancelation of $\ca$ implies that $\sigma^\alpha = \sigma^\alpha \tau^\alpha \sigma^\alpha$ and $ \tau^\alpha=\tau^\alpha \sigma^\alpha\tau^\alpha$. So the set
$$\mathcal{S}(\ca):=\langle \sigma^\alpha,\tau^\alpha :\alpha \in \ca rgle \subseteq \mathfrak{I}(\ca)$$
with multiplication and involution induced from $\mathfrak{I}(\ca)$ is an inverse semigroup \cite[Lemma 2.3]{ort20}.
\end{defn}

In order to define the inverse semigroup $\mathcal{T}(\ca)$, we need an LCSC to be finitely aligned. Before that, we should recall some terminology. Two elements $s,t$ of an inverse semigroup $S$ are said to be \emph{compatible} if both $s^* t$ and $st^*$ are idempotents. \cite[Proposition 1.2.1]{law98} implies that if $f_1,\ldots,f_n$ are pairwise compatible elements in $\mathfrak{I}(X)$, then the \emph{least-upper bound} $\bigvee_{i=1}^n f_i$ (according to the natural partial order $\leq$ in $\mathfrak{I}(X)$ as an inverse semigroup) exists and belongs to $\mathfrak{I}(X)$.

\begin{defn}[{\cite[Definition 2.7]{ort20}}]
If $\ca$ is a finitely aligned LCSC, we define
$$\mathcal{T}(\ca)=\left\{\bigvee_{i=1}^n \tau^{\alpha_i}\sigma^{\beta_i}: \{\tau^{\alpha_i}\sigma^{\beta_i}\}_{i=1}^n\subseteq \mathcal{S}(\ca) ~~ \mathrm{are ~ pairwise ~ compatible} \right\}$$
as a subset of $\mathfrak{I}(\ca)$.
\end{defn}

\cite[Lemma 2.8]{ort20} says that $\mathcal{T}(\ca)$ is the smallest finitely complete, finitely distributive inverse semigroup containing $\mathcal{S}(\ca)$. In the sequel, we will consider the inverse semigroups $\mathcal{S}(\La \rtimes^\varphi G)$ and $\mathcal{T}(\La \rtimes^\varphi G)$ of a Zappa-Sz$\acute{\mathrm{e}}$p product $\La \rtimes^\varphi G$.


\section{An inverse semigroup associated to $(G,\Lambda)$}\label{sec3}

Let $(G,\La,\varphi)$ be a self-similar $k$-graph. The aim of this section is to associate an appropriate inverse semigroup $\sg$ to $(G,\La,\varphi)$ such that the tight $C^*$-algebra of $\sg$ is isomorphic to $\og$ (Theorem \ref{thm4.7}). In light of Proposition \ref{prop2.6}, in order to define the desired $\sg$ we only consider triples $(\mu,g,\nu)$ in $\La\times G \times \La$ satisfying $s(\mu)=g \cdot s(\nu)$.

{\bf Property $(\perp)$}: Two triples $(\mu,g,\nu)$ and $(\xi,h,\eta)$ are called {\it orthogonal}, denoted by $(\mu,g,\nu)\perp (\xi,h,\eta)$, if $\Lam(\mu,\xi)=\Lam(\nu,\eta)=\emptyset$.

\begin{defn}\label{defn3.1}
Let $\sg$ be the collection of all finite sets $F$ containing pairwise orthogonal triples $(\mu,g,\nu)$ satisfying $s(\mu)=g \cdot s(\nu)$. We equip $\sg$ with the multiplication
\begin{equation}\label{eq3.1}
EF:=\bigcup_{\substack{(\mu,g,\nu)\in E \\ (\xi,h,\eta)\in F}} \left\{ (\mu(g\cdot\alpha), \varphi(g,\alpha)\varphi(h,h^{-1}\cdot\beta), \eta(h^{-1}\cdot\beta)):(\alpha,\beta)\in\Lam(\nu,\xi) \right\}
\end{equation}
and the inverse
$$F^*:=\left\{(\nu,g^{-1},\mu):(\mu,g,\nu)\in F\right\}$$
for all $E,F\in \sg$.
\end{defn}

Given any $E,E\in \sg$, orthogonality of elements in $E$, and also in $F$, implies that triples of $EF$ in (\ref{eq3.1}) are pairwise orthogonal. Furthermore, for $(\mu,g,\nu)\in E$, $(\xi,h,\eta)\in F$ and $(\alpha,\beta)\in\Lam(\nu,\xi)$, applying \cite[Lemma 3.5(i)]{li21} gives
\begin{align*}
\varphi(g,\alpha)\varphi(h,h^{-1}\cdot \beta)\cdot s(\eta(h^{-1}\cdot \beta)) & =\varphi(g,\alpha)\cdot (h\cdot s(h^{-1}\cdot \beta)) \\
& =g\cdot s(\beta)\\
& =g \cdot s(\xi\beta)\\
& =g \cdot s(\nu\alpha)\\
& =s(g \cdot \alpha)\\
& =s(\mu(g\cdot \alpha)).
\end{align*}
Hence, $EF$ defined in (\ref{eq3.1}) is an element of $\sg$ as well, so multiplication on $\sg$ is well-defined.

Note that the definition of $\sg$ in Definition \ref{defn3.1} substantially generalizes that of $\mathcal{S}_\La$ in \cite[Section 4]{far05} to self-similar $k$-graphs. In particular, if $G=\{e_G\}$ is the trivial group, then $\sg=\mathcal{S}_\La$.

\begin{rem}\label{rem3.2}
The property ($\perp$) says that for every $(\mu,g,\nu),(\xi,h,\eta)\in F$ we have
$$\Lam(\mu,\xi)\neq \emptyset ~~ \Longleftrightarrow ~~ (\mu,g,\nu)=(\xi,h,\eta) ~~ \Longleftrightarrow ~~ \Lam(\nu,\eta)\neq \emptyset.$$
\end{rem}

\begin{lem}\label{lem3.3}
The multiplication on $\sg$ is associative.
\end{lem}

\begin{proof}
It suffices for every singletons $\{(\la,g,\mu)$\}, $\{(\xi,h,\eta)\}$, and $\{(\tau,f,\omega)\}$ in $\sg$ to prove
$$\big[\{ (\la,g,\mu)\}\{(\xi,h,\eta)\}\big] \{(\tau,f,\omega)\}= \{(\la,g,\mu)\} \big[\{(\xi,h,\eta)\} \{(\tau,f,\omega)\}\big],$$
that is, by Definition \ref{defn3.1},
\begin{align}\label{eq3.2}
\bigcup_{\substack{(\alpha,\beta) \in \Lam(\mu,\xi)\\ (\nu,\zeta)\in  \Lam(\eta(h^{-1}\cdot\beta),\tau)}}  \big\{\big(\la(g\cdot\alpha) & (\varphi(g,\alpha)\varphi(h,h^{-1}\cdot\beta)\cdot\nu),
\nonumber \\
\varphi(\varphi(g,\alpha) & \varphi(h,h^{-1}\cdot\beta),\nu)\varphi(f,f^{-1}\cdot\zeta), \omega(f^{-1}\cdot\zeta) \big) \big\}
\nonumber\\
\nonumber \\
=\bigcup_{\substack{(\nu',\zeta') \in \Lam(\eta,\tau)\\ (\alpha',\beta')\in  \Lam(\mu,\xi(h\cdot\nu'))}}  \big\{\big(\la(g\cdot\alpha') & ,\varphi(g,\alpha')\varphi[\varphi(h,\nu')\varphi(f,f^{-1}\cdot\zeta'),
\nonumber\\
(\varphi(h,\nu') \varphi(f,f^{-1}\cdot\zeta'))^{-1} & \cdot\beta'], \omega(f^{-1}\cdot\zeta') ((\varphi(h,\nu')\varphi(f,f^{-1}\cdot\zeta'))^{-1}\cdot\beta')\big)\big\}.
\end{align}
Fix a triple in the left-hand side set of (\ref{eq3.2}) above. Since $(\alpha,\beta) \in \Lam(\mu,\xi)$ and $(\nu,\zeta)\in  \Lam(\eta(h^{-1}\cdot\beta),\tau)$, we have
\begin{equation}\label{eq3.3}
\mu\alpha=\xi \beta \hspace{10mm} \mathrm{and} \hspace{10mm} \eta(h^{-1}\cdot\beta)\nu=\tau\zeta
\end{equation}
with
\begin{equation}\label{eq3.4}
d(\mu\alpha)=d(\mu)\vee d(\xi) \hspace{5mm} \mathrm{and} \hspace{5mm} d(\eta(h^{-1}\cdot\beta)\nu)=d(\tau\zeta)=d(\eta(h^{-1}\cdot\beta))\vee d(\tau).
\end{equation}
If we define $\nu':=((h^{-1}\cdot\beta)\nu)(0,d(\eta)\vee d(\tau)-d(\eta))$ and $\zeta':=\zeta(0,(d(\eta)\vee d(\tau))-d(\tau))$, then $(\nu',\zeta')\in \Lam(\eta,\tau)$. Set also $\alpha':=\alpha(\varphi(h,h^{-1}\cdot\beta)\cdot\nu)$. Hence
\begin{align*}
\mu \alpha'&=\mu \alpha(\varphi(h,h^{-1}\cdot\beta)\cdot\nu) \\
&=\xi \beta(\varphi(h,h^{-1}\cdot\beta)\cdot\nu) \hspace{5mm} (\mathrm{by} ~ (\ref{eq3.3}))\\
&=\xi(h.((h^{-1}\cdot\beta)\nu))\\
&=\xi(h\cdot\nu')[(h\cdot(h^{-1}\cdot\beta)\nu)(d(\nu'),d((h^{-1}\cdot\beta)\nu))]\\
&=\xi(h\cdot\nu')[\varphi(h,\nu')\cdot\big(((h^{-1}\cdot\beta)\nu)(d(\nu'),d((h^{-1}\cdot\beta)\nu))\big)],
\end{align*}
meaning that $\mu\alpha'=\xi(h\cdot\nu')\beta'$ where
$$\beta':= \varphi(h,\nu')\cdot\big(((h^{-1}\cdot\beta)\nu)(d(\nu'),d((h^{-1}\cdot\beta)\nu))\big).$$

Observe that in fact we have $(\alpha',\beta')\in \Lam(\mu,\xi(h\cdot\nu'))$ because
\begin{align*}
d(\mu\alpha')&=d(\mu\alpha)+ d(\nu)\\
&=(d(\mu)\vee d(\xi)) + d(\nu) \hspace{5mm} (\mathrm{by} ~ (\ref{eq3.4}))\\
&=(d(\mu)\vee d(\xi))+\big(d(\eta(h^{-1}\cdot\beta))\vee d(\tau)\big)- d(\eta(h^{-1}\cdot\beta))  \hspace{5mm} (\mathrm{by} ~ (\ref{eq3.4}))\\
&=(d(\mu)\vee d(\xi))+(0\vee (d(\tau)-d(\eta)-d(h^{-1}\cdot\beta)))\\
&=(d(\mu)\vee d(\xi))+(0\vee (d(\tau)-d(\eta)-(d(\mu)\vee d(\xi))+d(\xi)))  \hspace{3mm} {\scriptstyle (\mathrm{as}~ d(h^{-1}\cdot\beta)=d(\beta))}\\
&=d(\mu)\vee d(\xi) \vee (d(\tau)-d(\eta)+d(\xi))\\
&=d(\mu)\vee (d(\xi)+(d(\eta)\vee d(\tau))-d(\eta)) \\
&=d(\mu)\vee (d(\xi)+ d(\nu'))  \hspace{30mm} {\scriptstyle (\mathrm{beacuse}~(\nu',\zeta')\in\Lam(\eta,\tau))} \\
&=d(\mu)\vee d(\xi(h\cdot\nu')).
\end{align*}
It is clear that
\begin{equation}\label{eq3.5}
\la(g\cdot\alpha)(\varphi(g,\alpha)\varphi(h,h^{-1}\cdot\beta)\cdot\nu)=\la(g \cdot (\alpha \varphi(h,h^{-1}\cdot\beta)\cdot\nu))=\la (g\cdot\alpha')
\end{equation}
and also
\begin{align}\label{eq3.6}
\omega(f'\cdot\zeta')[\varphi(h,\nu')\varphi(f,f^{-1}\cdot\zeta')]^{-1}\cdot\beta'&=\omega(f^{-1}\cdot\zeta') [\varphi(f,f^{-1}\cdot\zeta')^{-1}\varphi(h,\nu')^{-1}]\cdot\beta' \nonumber\\
&=\omega(f^{-1}\cdot\zeta')[\varphi(f^{-1},\zeta')\varphi(h^{-1},h^{-1}\cdot\nu')\cdot\beta'] \nonumber\\
&=\omega[f^{-1}.(\zeta'(\varphi(h,\nu')^{-1}\cdot\beta'))] \nonumber \\
&=\omega (f^{-1}.(\zeta'[((h^{-1}\cdot\beta)\nu)(d(\nu'),d((h^{-1}\cdot\beta)\nu))])) \nonumber \\
&=\omega(f^{-1}\cdot\zeta).
\end{align}

Furthermore, we can compute
\begin{align}\label{eq3.7}
\varphi(\varphi(g,\alpha)\varphi(h,h^{-1} & \cdot\beta), \nu)\varphi(f,f^{-1}\cdot\zeta)\\
\hspace{-7mm}
{\scriptstyle  (\mathrm{by ~ (2)(a) ~ of ~ Def.} ~ \ref{defn2.3})} \hspace{5mm} &=\varphi(\varphi(g,\alpha), \varphi(h,h^{-1}\cdot\beta)\cdot\nu)\varphi(\varphi(h,h^{-1}\cdot\beta),\nu)\varphi(f,f^{-1}\cdot\zeta) \nonumber\\
\hspace{-7mm}
{\scriptstyle (\mathrm{by  ~ (2)(c) ~ of ~ Def.} ~ \ref{defn2.3})} \hspace{5mm} &=[\varphi(g,\alpha(\varphi(h,h^{-1}\cdot\beta)\cdot\nu))]\varphi(h,(h^{-1}\cdot\beta)\nu) \varphi(f,f^{-1}\cdot\zeta)\nonumber\\
{\scriptstyle (\mathrm{by} ~ (\ref{eq2.1}))} \hspace{5mm}
&=\varphi(g,\alpha')\varphi(h,(h^{-1}\cdot\beta)\nu)\varphi(f^{-1},\zeta)^{-1}\nonumber\\
\hspace{-9mm}
{\scriptstyle (\mathrm{hint:} ~ \zeta=\zeta' (\varphi(h,\nu')^{-1}.\beta'))}\hspace{5mm} &=\varphi(g,\alpha')[\varphi(h,(\nu'\varphi(h,\nu')\cdot\beta'))]\varphi(f^{-1},\zeta)^{-1}\nonumber\\
\hspace{-7mm}
{\scriptstyle (\mathrm{by ~ (2)(c) ~ of ~ Def.} ~ \ref{defn2.3})} \hspace{5mm} &=\varphi(g,\alpha')[\varphi(\varphi(h,\nu'),\varphi(h,\nu')^{-1}\cdot\beta')][\varphi(\varphi(f^{-1},\zeta'),\varphi(h,\nu')^{-1}\cdot\beta')]^{-1} \nonumber\\
&=\varphi(g,\alpha')[\varphi(\varphi(h,\nu')^{-1},\beta')]^{-1}[\varphi(\varphi(f^{-1},\zeta'),\varphi(h,\nu')^{-1}\cdot\beta')]^{-1}\nonumber\\
&=\varphi(g,\alpha')[\varphi(\varphi(f^{-1},\zeta'),\varphi(h,\nu')^{-1}\cdot\beta') \varphi(\varphi(h,\nu')^{-1},\beta')]^{-1}\nonumber\\
&=\varphi(g,\alpha') [\varphi(\varphi(f,f^{-1}\cdot\zeta')^{-1},\varphi(h,\nu')^{-1}\cdot\beta') \varphi(\varphi(h,\nu')^{-1},\beta')]^{-1}\nonumber\\
\hspace{-7mm}
{\scriptstyle (\mathrm{by ~ (2)(a) ~ of ~ Def.} ~ \ref{defn2.3})} \hspace{5mm}
&=\varphi(g,\alpha') [\varphi(\varphi(f,f^{-1}\cdot\zeta')^{-1} \varphi(h,\nu')^{-1},\beta')]^{-1}\nonumber\\
\hspace{-7mm}
{\scriptstyle (\mathrm{by} ~ (\ref{eq2.1}))} \hspace{5mm}
&=\varphi(g,\alpha') [\varphi(\varphi(h,\nu')\varphi(f,f^{-1}\cdot\zeta'),\big(\varphi(h,\nu')\varphi(f,f^{-1}\cdot\zeta')\big)^{-1}\cdot\beta')]. \nonumber
\end{align}
Therefore, (\ref{eq3.5}), (\ref{eq3.6}) and (\ref{eq3.7}) imply that the set in the left-hand side of (\ref{eq3.2}) is contained in the right-hand side one. The reverse inclusion may be shown analogously, hence the result is proved.
\end{proof}

Idempotents play a key role to analyze inverse semigroups (in particular, to construct the associated groupoids of germs). In Corollary \ref{cor3.5} below we will see that the idempotents of $\sg$ coincides to those of $\mathcal{S}_\La$ the inverse semigroup of $\La$ in \cite{far05}.

\begin{lem}\label{lem3.4}
For any $F\in \sg$ we have
$$FF^*=\{(\mu,e_G,\mu):(\mu,g,\nu)\in F\}.$$
\end{lem}

\begin{proof}
Recall that $F^*=\{(\nu,g^{-1},\mu): (\mu,g,\nu)\in F\}$. If $(\mu,g,\nu),(\xi,h,\eta)\in F$ are distinct, then Remark \ref{rem3.2} says  $\Lam(\nu, \eta)=\emptyset$, while in the case $(\xi,h,\eta)=(\mu,g,\nu)$ we have $\Lam(\nu,\eta)=\Lam(\nu,\nu)=\{(s(\nu),s(\nu))\}$. Thus (\ref{eq3.1}) follows the result.
\end{proof}

\begin{cor}\label{cor3.5}
$\es$ is the collection of finite subsets $F$ of $\{(\mu,e_G,\mu): \mu\in \La\}$ satisfying the property $\Lam(\mu,\nu)=\emptyset$ for every distinct $(\mu,e_G,\mu), (\nu,e_G,\nu) \in F$.
\end{cor}

Note that Corollary \ref{cor3.5} is analogous to \cite[Lemma 2.9]{ort20} for the inverse semigroup $\mathcal{T}(\La\rtimes^{\varphi} G)$, which says that every idempotent $f$ in $\mathcal{T}({\La\rtimes^{\varphi} G})$ is of the form
$$f=\bigvee_{i=1}^n \tau^{(\mu_i,g_i)}\sigma^{(\mu_i,g_i)}=\bigvee_{i=1}^n \tau^{(\mu_i,e_G)}\sigma^{(\mu_i,e_G)}$$
for some $\mu_1,\ldots,\mu_n\in \La$.

\begin{prop}
$\sg$, defined in Definition \ref{defn3.1}, is an inverse semigroup with the zero $0:=\emptyset$.
\end{prop}

\begin{proof}
By Lemma \ref{lem3.3}, the multiplication in $\sg$ is associative. Also, Lemma \ref{lem3.4} implies $FF^* F=F$ and $F^* F F^*=F^*$. So, it remains to prove that $F^*$ is the unique element of $\sg$ satisfying these last equalities.

To show this, we assume $E,F\in \sg$ with $FEF=F$ and $EFE=E$, and prove $E=F^*$. Fix an arbitrary triple $(\xi,h,\eta)\in E$. So, there exist $(\la,g,\mu),(\tau,f,\omega)\in F$ such that $\Lam(\mu,\xi)\neq \emptyset$ and $\Lam(\eta,\tau)\neq \emptyset$; because otherwise we will have either $\{(\xi,h,\eta)\}F=0$ or $F\{(\xi,h,\eta)\}=0$, and hence $EFE\subseteq E\setminus\{(\xi,h,\eta)\}$ in each case, a contradiction. Then
$$0 \neq \{(\la,g,\mu)\}\{(\xi,h,\eta)\}\{(\tau,f,\omega)\}\subseteq FEF=F,$$
which is
\begin{multline}\label{eq3.8}
\hspace{0mm} \bigcup_{\substack{(\nu,\zeta)\in \Lam(\eta(h^{-1}\cdot\beta),\tau)\\ (\alpha,\beta)\in \Lam(\mu,\xi)}}
\big\{ (\la(g\cdot\alpha)(\varphi(g,\alpha)\varphi(h,h^{-1}\cdot\beta)\cdot\nu), \\
\varphi(\varphi(g,\alpha)\varphi(h,h^{-1}\cdot\beta),\nu) \varphi(f,f^{-1}\cdot\zeta),\omega(f^{-1}\cdot\zeta)) \big\} \subseteq F.
\end{multline}

But, since $(\la,g,\mu),(\tau,f,\omega)\in F$, property ($\perp$) for $F$ implies that each element of the set in (\ref{eq3.8}) must be equal to $(\la,g,\mu)$, and also to $(\tau,f,\omega)$. So, $(\la,g,\mu)=(\tau,f,\omega)$ in particular. Moreover, it follows $\Lam(\mu,\xi)=\{(s(\mu),s(\xi))\}$ and $\Lam(\eta,\tau)=\{s(\eta),s(\tau)\}$, meaning that $\mu=\xi$ and $\eta=\tau=\la$ respectively. Consequently, every element in the left set of (\ref{eq3.8}) is of the form $(\la,ghf,\omega)$. Furthermore, since $(\la,g,\mu)$ lies in $F$ too, property ($\perp$) forces again $(\la,g,\mu)=(\la,ghf,\omega)$. Therefore, $h=f^{-1}=g^{-1}$ and $(\xi,h,\eta)=(\mu,g^{-1},\la)\in F^*$, concluding $E\subseteq F^*$. The reverse inclusion may be shown analogously. The proof is completed.
\end{proof}


\section{Tight $C^\ast$-algebra of $\sg$}\label{sec4}

The tight $C^*$-algebra of an inverse semigroup $S$ with $0$, denoted by $C^*_{\mathrm{tight}}(S)$, is introduced in \cite{exe08}, which is generated by a universal tight representation of $S$. In this section, we prove that $C^*_{\mathrm{tight}}(\sg)\cong \og$. Recall from \cite[Definition 13.1]{exe08} that a representation $\pi:S\rightarrow B(H)$ on a Hilbert space $H$ is called \emph{tight} if for every $X,Y\subseteq \mathcal{E}(S)$ and every finite cover $Z$ for the set
$$\mathcal{E}(S)^{X,Y}:=\left\{e\in \mathcal{E}(S): e\leq f ~~~ \forall f\in X, ~ \mathrm{and} ~ ef'=0 ~~~ \forall f'\in Y\right\}$$
(see Definition \ref{defn4.1} below), one has
$$\bigvee_{e\in Z} \pi(e)=\bigwedge_{f\in X} \pi(f) \wedge \bigwedge_{f'\in Y}(1-\pi(f')).$$
Also, we say that $\pi$ is a \emph{universal tight representation} if for every tight representation $\phi:S\rightarrow \mathcal{A}$, there exists a $*$-homomorphism $\psi: \pi(S)\rightarrow \mathcal{A}$ such that $\psi \circ \pi=\phi$. Note that, in the proof of Theorem \ref{thm4.7}, we use the notion of \emph{cover-to-join} for a representation $\pi$ of $\sg$ which is equivalent to the tightness.

\begin{defn}[\cite{exe16}]\label{defn4.1}
Let $(G,\La)$ be a self-similar $k$-graph and $\sg$ the inverse semigroup of $(G,\La)$ as in Definition \ref{defn3.1}. Fix some $F\in \sg$.
\begin{enumerate}[(1)]
  \item The \emph{principal ideal} of $\es$ generated by $F$ is defined by
$$\mathcal{J}_F:=\{E\in \es: E\leq F\},$$
which is equal to $\{E\in \es: FE=E\}$. (Recall that, for any $E\in \es$, we have $E\leq F$ if and only if $FE=E$.)
  \item A subset $\mathcal{C}\subseteq \es$ is called an \emph{outer cover} for $\mathcal{J}_F$ (or simply for $F$) if for every $E\in \mathcal{J}_F$ there exists some $E'$ in $\mathcal{C}$ such that $EE'\neq 0$.
  \item A subset $\mathcal{C}\subseteq \es$ is called a \emph{cover} for $\mathcal{J}_F$ (or simply for $F$) if $\mathcal{C}$ is an outer cover and $\mathcal{C}\subseteq \mathcal{J}_F$.
\end{enumerate}
\end{defn}

While we work with principal ideals of $\es$, the next two lemmas will make our arguments easier. Before them, let us state a definition.

\begin{defn}\label{defn4.2}
Let $(G,\La)$ be a self-similar $k$-graph. We say that a path $\mu\in \La$ is \emph{strongly fixed by $g\in G$} if $g\cdot \mu=\mu$ and $\varphi(g,\mu)=e_G$. Following \cite{exe17}, if for every $g\in G\setminus\{e_G\}$ there are no paths in $\La$ strongly fixed by $g$, then $(G,\La)$ is called \emph{pseudo free}.
\end{defn}

\begin{lem}\label{lem4.3}
Given any singletons $\{(\mu,e_G,\mu)\}\in \es$ and $\{(\xi,g,\eta)\}\in \sg$, then $\{(\mu,e_G,\mu)\}\leq \{(\xi,g,\eta)\}$ in $\sg$ if and only if the following hold:

\begin{enumerate}[(1)]
  \item $\eta=\xi$,
  \item $\mu=\xi \beta$ for some $\beta\in s(\mu)\La$, and
  \item $\beta$ is strongly fixed by $g$.
\end{enumerate}
\end{lem}

\begin{proof}
The assumption $\{(\mu,e_G,\mu)\}\leq \{(\xi,g,\eta)\}$ implies
\begin{align*}
\{(\mu,e_G,\mu)\}&=\{(\mu,e_G,\mu)\}\{(\xi,g,\eta)\} \\
&=\bigcup_{(\alpha,\beta)\in \Lam(\mu,\xi)} \{(\mu\alpha, \varphi(g,g^{-1}\cdot\beta),\eta(g^{-1}\cdot\beta))\},
\end{align*}
so $\alpha=s(\mu)$, $\eta(g^{-1}\cdot\beta)=\mu$ and $\varphi(g,g^{-1}\cdot\beta)=e_G$ for all $(\alpha,\beta)\in \Lam(\mu,\xi)$. Consequently, we obtain $\xi\beta=\mu=\eta(g^{-1}\cdot\beta)$, and hence $\beta=g^{-1}\cdot\beta$ and $\eta=\xi$ by the unique factorization property. These prove the ``only if" part. The reverse implication follows from the multiplication in (\ref{eq3.1}).
\end{proof}

\begin{lem}\label{lem4.2}
Let $F\in \sg$ and $E\in \es$. Then $E\leq F$ if and only if for any $(\mu,e_G,\mu)\in E$, there exists a (unique) triple of the form $(\xi,g,\xi)\in F$ such that $\{(\mu,e_G,\mu)\}\leq \{(\xi,g,\xi)\}$.
\end{lem}

\begin{proof}
If $E\leq F$ then $FE=E$. In particular, by eq. (\ref{eq3.1}), for any triple $(\mu,e_G,\mu)$ in $E$, there exists $(\xi,g,\eta)\in F$ such  that
$$\{(\mu,e_G,\mu)\}\{(\xi,g,\eta)\}=\{(\mu,e_G,\mu)\}.$$
But we must have $\eta=\xi$ by Lemma \ref{lem4.3}, and hence $\{(\mu,e_G,\mu)\}\leq \{(\xi,g,\xi)\}$.

For the converse, it suffices to write
\begin{align*}
FE&=\bigcup_{(\mu,e_G,\mu)\in E} F \{(\mu,e_G,\mu)\}\\
&=\bigcup_{\substack{(\mu,e_G,\mu)\in E \\ (\xi,g,\xi)\in F}} \{(\xi,g,\xi)\} \{(\mu,e_G,\mu)\}\\
&=\bigcup_{(\mu,e_G,\mu)\in E} \{(\mu,e_G,\mu)\}\\
&=E,
\end{align*}
as desired.
\end{proof}

{\bf Notation.} In order to ease the notation, we will denote each singleton of the form $\{(\mu,e_G,\mu)\}$ in $\es$ by $\iota_\mu$.

\begin{rem}\label{rem4.3}
Note that if $F$ is an idempotent in $\es$, then $\mathcal{J}_F=\{FE: E\in \es\}$. Moreover, if $F$ is of the form $F=\iota_{\mu_1}\cup \cdots \cup \iota_{\mu_l}$, then $\iota_{\mu_i}$'s are pairwise orthogonal and we have
$$\mathcal{J}_F=\left\{E=E_1\cup \cdots \cup E_l: E_i\in \mathcal{J}_{\iota_{\mu_i}} ~ \mathrm{for} ~ 1\leq i\leq l \right\}.$$
Hence, every cover (outer cover) $\mathcal{C}$ for $\mathcal{J}_F$ can be written as $\mathcal{C}=\mathcal{C}_1 \cup \cdots \cup \mathcal{C}_l$, where each $\mathcal{C}_i$ is a cover (an outer cover, respectively) for $\mathcal{J}_{\iota_{\mu_i}}$. We will use these facts in the proof of Theorem \ref{thm4.7} below.
\end{rem}

We need also the next lemma to prove Theorem \ref{thm4.7}.

\begin{lem}\label{lem4.6}
Let $(G,\La)$ be a self-similar $k$-graph. For every $\iota_\mu=\{(\mu,e_G,\mu)\}\in \es$, where $\mu\in \La$, the following statements hold:
\begin{enumerate}[(1)]
  \item
  \begin{multline*}
  \mathcal{J}_{\iota_\mu}=\big\{E\subseteq_{\mathrm{finite}}  \{(\mu\tau, e_G, \mu\tau): \tau\in s(\mu)\La\}: \Lam(\tau,\tau')= \emptyset ~ \\
  \mathrm{for ~ distinct}  ~ (\mu\tau, e_G, \mu\tau),(\mu\tau', e_G, \mu\tau') \in E \big\}.
  \end{multline*}
  \item $\mathcal{C}$ is a cover for $\iota_\mu$ if and only if the set $\left\{\iota_\la: (\la,e_G,\la)\in \bigcup_{E\in \mathcal{C}}E\right\}$ is, and if and only if $\mathcal{C}'=\{\iota_\tau: (\mu\tau,e_G,\mu\tau)\in \bigcup_{E\in \mathcal{C}} E\}$ is a cover for $\iota_{s(\mu)}$.
  \item $\mathcal{C}$ is a cover for $\iota_\mu$ if and only if the set
  $$\Omega_{\mathcal{C}}:=\left\{\tau:(\mu\tau,e_G,\mu\tau)\in \bigcup_{E\in \mathcal{C}} E\right\}$$
   is exhaustive in $s(\mu)\La$.
\end{enumerate}
\end{lem}

\begin{proof}
(1). For two elements $\iota_\la, \iota_\mu\in \es$, Lemma \ref{lem4.3} implies that $\iota_\la\leq \iota_\mu$ if and only if $\la=\mu\tau$ for some $\tau\in s(\mu)\La$. So, every $E\in \es$ with $E\leq \iota_\mu$ must be a finite subset of $\{(\mu\tau, e_G, \mu\tau): \tau\in s(\mu)\La\}$ satisfying $\Lam(\mu\tau,\mu\tau')=\emptyset$ for distinct $(\mu\tau, e_G, \mu\tau),(\mu\tau', e_G, \mu\tau') \in E$. But, since $\Lam(\mu\tau,\mu\tau')=\Lam(\tau,\tau')$, the right-hand side of (1) is included in $\mathcal{J}_{\iota_\mu}$. The proof of the reverse inclusion is analogous.

(2). Using part (1), we have $E\in \mathcal{J}_{\iota_{\mu}}$ if and only if the $E'=\{(\tau,e_G,\tau): (\mu\tau,e_G,\mu\tau)\in E\}$ belongs to $\mathcal{J}_{\iota_{s(\mu)}}$. Hence statement (2) follows from the fact that for every $E\in \mathcal{C}$ and $F\in \mathcal{J}_{\iota_\mu}$,
\begin{align*}
EF\neq 0  \hspace{5mm} \Longleftrightarrow \hspace{5mm} \exists ~ (\mu\tau,e_G,\mu\tau) & \in E ~~\mathrm{such ~ that~} \iota_{\mu\tau}F\neq 0 ~~\mathrm{in}~~\sg \\
\Longleftrightarrow \hspace{5mm} \exists ~ (\mu\tau,e_G,\mu\tau) & \in E ~~\mathrm{such ~ that~} \iota_{\tau}F'\neq 0,  ~~\mathrm{where} \\
&~~ F':=\{(\tau',e_G,\tau'):(\mu\tau',e_G,\mu\tau') \in F\}.
\end{align*}

(3). By part (2), $\mathcal{C}$ is a cover for $\iota_\mu$ if and only if
$$\mathcal{C}'=\left\{\iota_\tau:(\mu\tau,e_G,\mu\tau)\in  \cup_{E\in \mathcal{C}}E\right\}$$
is a cover for $\iota_{s(\mu)}$. Therefore, combining this with
\begin{align*}
  \la\in s(\mu)\La \hspace{5mm} &\Longleftrightarrow \hspace{5mm} \iota_\la\in \mathcal{J}_{\iota_{s(\mu)}}\\
   &\Longleftrightarrow \hspace{5mm} \exists \iota_{\tau}\in \mathcal{C}' ~~\mathrm{such ~ that~} \iota_{\la}\iota_\tau\neq 0 \\
   &\Longleftrightarrow \hspace{5mm} \exists \tau\in \Omega_{\mathcal{C}} ~~\mathrm{such ~ that~} \Lam(\la,\tau)\neq 0,
\end{align*}
concludes statement (3).
\end{proof}

\begin{thm}\label{thm4.7}
Let $(G,\La)$ be a self-similar $k$-graph as in Definition \ref{defn2.3}. Then the map $\pi:\sg \rightarrow \og$ defined by $\pi(0)=0$ and
\begin{equation}\label{eq4.1}
\pi(F)=\sum_{(\mu,g,\nu)\in F}s_\mu u_{s(\mu),g} s_\nu^* \hspace{10mm} (F\in \sg)
\end{equation}
is a universal tight representation (in the sense of \cite[Definition 13.1]{exe08}). Therefore, we have $C^*(\gt)\cong C^*_{\mathrm{tight}}(\sg) \cong \og$.
\end{thm}

\begin{proof}
It is straightforward to show that $\pi$ is a $*$-homomorphism. So, we first prove the tightness of $\pi$. According to \cite[Theorem 6.1]{exe21} or \cite[Corollary 2.3]{don14}, it suffices to verify that $\pi$ is a \emph{cover-to-joint representation} in the sense: if $\mathcal{C}$ is a finite cover for $F\in \es$, then
\begin{equation}\label{eq4.2}
\pi(F)=\bigvee_{E\in \mathcal{C}} \pi(E).
\end{equation}
Let $F$ be an element in $\es$. In light of Remark \ref{rem4.3}, for every $E\in \mathcal{J}_F$, there are $\iota_{\mu_1},\ldots, \iota_{\mu_l}\subseteq F$ and $E_i\in \mathcal{J}_{\iota_{\mu_i}}\subseteq \mathcal{J}_F$ for $1\leq i\leq l$, such that $E=E_1 \sqcup \cdots \sqcup E_n$. Hence, without loss of generality, we may prove (\ref{eq4.2}) only for singletons of the form $F=\iota_\mu$ with $\mu\in \La$.

So, fix some $F=\iota_\mu$ in $\es$ and suppose that $\mathcal{C}$ is a finite cover for $\iota_\mu$. Let us first consider the case $\mu=v\in \La^0$. Then every $E\leq \iota_v$ is a subset of $\bigcup_{\tau \in v\La}\iota_\tau$ by Lemma \ref{lem4.6}(1). If we define $\Omega:=\{\tau:(\tau,e_G,\tau) \in \bigcup_{E\in \mathcal{C}}E\}$, then Lemma \ref{lem4.6}(3) says that $\Omega$ is exhaustive, and by (CK4) of Definition \ref{defn2.2} we get
$$\prod_{\tau\in \Omega}(s_v-s_\tau s_\tau^*)=0$$
$$\Longrightarrow \hspace{5mm}  s_v=\bigvee_{\tau\in \Omega} s_\tau s_\tau^* =\bigvee_{\tau\in \Omega} \pi(\iota_\tau)=\bigvee_{E\in \mathcal{C}} \pi(E),$$
proving (\ref{eq4.2}) in this case.

For the general case, assume that $\mu\in \La$ is given and $\mathcal{C}$ is a finite cover for $\iota_\mu$. Then Lemma \ref{lem4.6}(2) implies that $\mathcal{C}'=\{\iota_\tau:(\mu\tau,e_G,\mu\tau)\in \bigcup_{E\in \mathcal{C}} E\}$ is a finite cover for $\iota_{s(\mu)}$, and by the previous case, we get
\begin{align*}
\pi(\iota_\mu)=s_\mu s_{s(\mu)}s_\mu^* & = s_\mu(\bigvee_{\tau\in \mathcal{C}'}s_\tau s_\tau^*) s_\mu^*\\
&=\bigvee_{\iota_{\mu\tau} \subseteq \cup_{E\in \mathcal{C}} E} s_{\mu\tau} s_{\mu \tau}^*= \bigvee_{E\in \mathcal{C}}\pi(E)
\end{align*}
concluding that $\pi$ is a cover-to-joint (so tight) representation.

We now prove that $\pi$ is a \emph{universal} tight representation. To do this, assume $\phi:\sg \rightarrow \mathcal{A}$ is a tight representation for $\sg$. Define $S_\mu:=\phi(\{(\mu, e_G,s(\mu))\})$ and $U_{v,g}:=\phi(\{(v,g,g^{-1}\cdot v)\})$ for all $\mu\in \La$, $v\in \La^0$, and $g\in G$. We claim that $\{S_\mu,U_{v,g}\}$ is a $(G,\La)$-family in $\mathcal{A}$. Indeed, all relations in Definition \ref{defn2.4} follow directly from the multiplication and inverse in $\sg$ and that $\phi$ is a $*$-homomorphism; but only the (CK4) property for $\{S_\mu:\mu\in \La\}$ is not trivial. To verify the (CK4), let $\Omega\subseteq v\La$ be a finite exhaustive set. Then $\mathcal{C}=\{\iota_\tau:\tau\in \Omega\}$ is a finite cover for $\iota_v$ by Lemma \ref{lem4.6}(3), and the tight property of $\phi$ yields
$$S_v=\phi(\iota_v)=\bigvee_{\iota_\tau\in \mathcal{C}}\phi (\iota_\tau)=\bigvee_{\tau\in \Omega}S_\tau S_\tau^*,$$
or equivalently
$$\prod_{\tau\in \Omega}(S_v-S_\tau S_\tau^*)=0.$$
This proves (CK4). Therefore, the universality of $\og$ gives a $*$-homomorphism $\psi :\og\rightarrow \mathcal{A}$ such that $\psi(s_\mu)=S_\mu$ and $\psi(u_{v,g})=U_{v,g}$ for all $\mu\in \La$, $v\in \La^0$ and $g\in G$. Consequently, $\psi\circ \pi=\phi$ and $\pi$ is a universal tight representation of $\sg$.

For the second statement, it suffices to notice that the linear span of $\mathrm{rang}(\pi)$ is dense in $\og$ by Proposition \ref{prop2.6}, and hence $\pi$ induces a $*$-isomorphism from $C^*_{\mathrm{tight}}(\sg)$ onto $\og$ by \cite[Theorem 13.3]{exe08}. Since $C^*(\gt)\cong C^*_{\mathrm{tight}}(\sg)$ by \cite[Corollary 10.16]{exe08}, the proof is completed.
\end{proof}

\begin{rem}\label{rem4.8}
Let $(G,\La)$ be a self-similar $k$-graph over a row-finite, source-free $k$-graph $\La$. According to \cite[Definition 3.9]{li21-ideal}, let $C^*_u(G,\La)$ be the universal unital $C^*$-algebra generated by a family
$$\{s_\mu:\mu\in \La\} \cup \{u_g:g\in G\}$$
satisfying the following relations:
\begin{enumerate}[(1)]
  \item $\{s_\mu:\mu\in \La\}$ is a $\La$-family,
  \item $u:G\rightarrow C^*_u(G,\La)$, defined by $g\mapsto u_g$, is a unitary $*$-representation of $G$ on $C^*_u(G,\La)$,
  \item $u_{gh}=u_g u_h$ for all $g,h\in G$, and
  \item $u_g s_\mu = s_{g\cdot \mu} u_{\varphi(g,\mu)}$ for all $g\in G$ and $\mu\in \La$.
\end{enumerate}
Then in \cite{li21-ideal} the $C^*$-algebra associated to $(G,\La)$ is defined by
$$\mathcal{A}_{G,\La}:=\overline{\mathrm{span}}\{s_\mu u_g s_\nu^*: g\in G,~ \mu,\nu\in \La,~ s(\mu)=g\cdot s(\nu)\}$$
as a $C^*$-subalgebra of $C^*_u(G,\La)$. However, one may replace $\og$ by $\mathcal{A}_{G,\La}$ in Theorem \ref{thm4.7} and mimic the argument to conclude $\mathcal{A}_{G,\La}\cong C^*_{\mathrm{tight}}(S_{G,\La}) \cong \og$. Therefore, Definition \ref{defn2.4} in Section \ref{sec2} and \cite[Definiton 3.9]{li21-ideal} define isomorphic $C^*$-algebras associated to $(G,\La)$ in this case (this follows also from Proposition \ref{prop7.2}).
\end{rem}


\section{Hausdorffness of $\gt$}\label{sec5}

In this section, we investigate the Hausdorff property of $\gt$ via properties of the underlying self-similar $k$-graph $(G,\La)$.

\begin{defn}\label{defn5.1}
Given any $g\in G$, we denote by $\mathbf{SF}_g$ the set of all finite paths in $\La$ strongly fixed by $g$ (see Definition \ref{defn4.2}), that is
$$\mathbf{SF}_g:=\{\tau\in \La:g\cdot\tau=\tau ~ \mathrm{and} ~~ \varphi(g,\tau)=e_G\}.$$
We say $\mathbf{SF}_g$ is \emph{locally exhausted} if for every $v\in \La^0$, there is a finite set $\mathcal{M}\subseteq v\mathbf{SF}_g$ such that
\begin{center}
``for every $\tau \in v\mathbf{SF}_g$, there is $\tau'\in \mathcal{M}$ with $\Lam(\tau,\tau')\neq \emptyset$".
\end{center}
\end{defn}

\begin{thm}\label{thm5.2}
Let $(G,\La)$ be a self-similar $k$-graph. Then $\gt$ is Hausdorff if and only if $\mathbf{SF}_g$ is locally exhausted for all $g\in G\setminus\{e_G\}$.
\end{thm}

\begin{proof}
Our main tool here is \cite[Theorem 3.16]{exe16} which says that $\gt$ is Hausdorff if and only if $\mathcal{J}_F$ has a finite cover in $\es$ for all $F\in \sg$. First, suppose that $\mathbf{SF}_g$ is locally exhausted for every $g\neq e_G$. Note that, by Lemma \ref{lem4.2}, every $E\in \mathcal{J}_F$ is a finite subset of
$$\left\{(\mu,e_G,\mu): \iota_\mu\in \mathcal{J}_{\{(\xi,g,\xi)\}} ~\mathrm{for ~ some~} (\xi,g,\xi)\in F\right\},$$
so it is equivalent to prove that every principal ideal of the form $\mathcal{J}_{\{(\xi,g,\xi)\}}$ has a finite cover in $\es$.

To do this, fix some $s=\{(\xi,g,\xi)\}$ in $\sg$. In the case that there are no nonzero idempotents $E\leq s$, then $\mathcal{J}_s=\emptyset$ and our cover will be trivial. Otherwise, using Lemma \ref{lem4.3}, we can assume that there exists an element of the form $\iota_{\xi\tau}\in \mathcal{J}_s$ where $\tau$ is strongly fixed for $g$. If $g=e_G$, then $\{(\xi,e_G,\xi)\}$ is a cover for $\mathcal{J}_s$. If $g\neq e_G$, by hypothesis, $s(\xi)\mathbf{SF}_g$ is exhausted by a finite set $\mathcal{M}\subseteq s(\xi)\mathbf{SF}_g$. We will prove that
$$\mathcal{C}=\big\{\{(\xi\tau,e_G,\xi\tau)\}:\tau\in \mathcal{M}\big\}$$
is a finite cover for $\mathcal{J}_s$. Indeed, for any $E\in \mathcal{J}_s$ and $(\xi\tau',e_G,\xi\tau')\in E$, since $\mathcal{M}$ exhausts $s(\xi)\mathbf{SF}_g$, there is $\tau\in \mathcal{M}$ such that $\Lam(\tau,\tau')\neq \emptyset$. Thus
$$\{(\xi\tau,e_G,\xi\tau)\} E= \{(\xi\tau,e_G,\xi\tau)\}\{(\xi\tau',e_G,\xi\tau')\}\neq 0,$$
where $\{(\xi\tau,e_G,\xi\tau)\}\in \mathcal{C}$. Therefore, in all cases $\mathcal{J}_s$ has a finite cover, concluding that $\gt$ is Hausdorff by \cite[Theorem 3.16]{exe16}.

Conversely, assume that $\gt$ is Hausdorff. Fix $e_G\neq g\in G$ and $v\in \La^0$. If $v\mathbf{SF}_g\neq \emptyset$ and $\tau\in v\mathbf{SF}_g$, then $g\cdot\tau=\tau$ and in particular
$$g\cdot v=g\cdot r(\tau)= r(g\cdot\tau)=r(\tau)=v.$$
So, $s=\{(v,g,v)\}$ belongs to $\sg$ and $\mathcal{J}_s$ has a finite cover by \cite[Theorem 3.16]{exe16}, say $\mathcal{C}$. Recall again from Lemma \ref{lem4.3} that every idempotent $E$ in $\mathcal{J}_s$ is a finite subset of $\{(\tau,e_G,\tau): \tau\in v\mathbf{SF}_g\}$. Define $\mathcal{M}:=\{\tau:(\tau,e_G,\tau)\in \cup_{E\in \mathcal{C}} E\}$ which is a finite set. For every $\tau\in v\mathbf{SF}_g$ we have $\iota_\tau \in \mathcal{J}_s$, and since $\mathcal{C}$ is a cover for $\mathcal{J}_s$, there exists $E\in \mathcal{C}$ such that
$$\iota_\tau E \neq 0,$$
or equivalently $\Lam(\tau,\tau')\neq \emptyset$ for some $(\tau,e_G,\tau)\in E$. This says that $v\mathbf{SF}_g$ is exhausted by $\mathcal{M}$, completing the proof.
\end{proof}

Another condition on a semigroup $S$ being equivalent to the Hausdorffness of $\mathcal{G}_{\mathrm{tight}}(S)$ is the
following: we say $S$ is a \emph{weak semilattice} if for every singly principle sets $s\downarrow, t\downarrow \subseteq S$, intersection $s\downarrow \cap t\downarrow$ is finitely generated as a lower set (see \cite[Theorem 4.20]{ste10} for example). So, Theorem \ref{thm5.2} implies that $\sg$ is a weak semilattice if and only if $\mathbf{SF}_g$ in $\La$ is locally exhausted for all $g\in G\setminus \{e_G\}$.

\begin{rem}
In the 1-graph setting, \cite[Proposition 5.8]{exe17} implies that a self-similar graph $(G,E)$ is pseudo free if and only if the associated inverse semigroup $\mathcal{S}_{G,E}$ is $E^*$-unitary. However, this result does not hold for our inverse semigroup of a self-similar $k$-graph (even for 1-graphs). The main reason comes from the initial differences between the definition of $\sg$ in Definition \ref{defn3.1} and that of $\mathcal{S}_{G,E}$ in \cite[Definition 4.1]{exe17}: elements of $\sg$ are sets of triples while those of $\mathcal{S}_{G,E}$ are just single triples. Indeed, even under the pseudo freeness assumption, we may have a non-idempotent $F\in \sg$ and an idempotent $E\in \mathcal{E}(\sg)$ with $E\subset F$ (hence $E\leq F$ in this case).
\end{rem}

Nevertheless, Theorem \ref{thm5.2} concludes:

\begin{cor}\label{cor5.4}
If a self-similar $k$-graph $(G,\La)$ is pseudo free (Definition \ref{defn4.2}), then the groupoid $\gt$ is Hausdorff.
\end{cor}

\begin{proof}
If $(G,\La)$ is pseudo free, there are no strongly fixed paths by $g$ for all $g\neq e_G$, and hence $\mathbf{SF}_g=\emptyset$. Therefore the result immediately follows from Theorem \ref{thm5.2}.
\end{proof}

Note that Corollary \ref{cor5.4} could be a translation of \cite[Corollary 4.13]{ort20} for the LCSC $\La\rtimes^{\varphi} G$. Indeed, a computation shows that a Zappa-Sz$\acute{\mathrm{e}}$p product $\La\rtimes^{\varphi} G$ is right cancelative if and only if the associated self-similar $k$-graph $(G,\La,\varphi)$ is pseudo free (cf. \cite[Remark 7.5]{ort20}). Since the finite alignment of $\La\rtimes^{\varphi} G$ is equivalent to that of $\La$, \cite[Proposition 3.6]{don14} and \cite[Corollary 4.13]{ort20} follow also Corollary \ref{cor5.4} above.


\section{The tight spectrum and tight groupoid}\label{sec6}

The boundary path space $\partial \La$ of a finitely aligned $k$-graph $\La$ was introduced in \cite[Section 5]{far05} equipped with a locally compact and Hausdorff topology (see also \cite{web11}). In fact, in \cite{far05} it is proved that $\partial \La$ is a closed subset of a larger topological space $X_\La$ (named as the path space of $\La$) \cite[Lemma 5.12]{far05}. In this section, we prove that the tight spectrum of $\sg$ is homeomorphic to the boundary space $\partial \La$. This correspondence is very helpful and useful for analyzing the tight groupoid $\gt$ and the associated $C^*$-algebra; in particular, to prove Theorems \ref{thm8.3} and \ref{thm9.6} below.

First, a lemma:

\begin{lem}\label{lem6.1}
Let $\mathcal{F}$ be an ultrafilter in $\mathcal{E}(\sg)$ (i.e. $\mathcal{F}\in \widehat{\mathcal{E}}_\infty(\sg)$). Then for each $F\in \mathcal{F}$, there exists a unique triple $(\la,e_G,\la)\in F$ such that $\iota_\la \Cap E$ for all $E\in \mathcal{F}$ (in the sense that $\iota_\la E=\iota_\la \wedge E\neq 0$). Moreover, such idempotent $\iota_\la$ belongs to $\mathcal{F}$.
\end{lem}

\begin{proof}
Fix some $F\in \mathcal{F}$. Since $\mathcal{F}$ is countable, we may arrange it as $\mathcal{F}=\{F_0,F_1,\ldots\}$ with $F_0=F$. For each $i\geq 0$, define $E_i=F_0\wedge F_1 \wedge \ldots \wedge F_i$. Then $E_i\in \mathcal{F}$ because $\mathcal{F}$ is a filter, and we have $E_0\geq E_1\geq \cdots$. So, it suffices to find a unique $(\la,e_G,\la)\in F$ such that $\iota_\la \Cap E_i$ for all $i\geq 1$.

On the contrary, assume that such $(\la,e_G,\la)$ in $F$ does not exist. Write $F=\{(\la_j,e_G,\la_j):1\leq j \leq l\}$. Then for any $1\leq j\leq l$, there is $E_{i_j} \in \mathcal{F}$ such that $\iota_{\la_j}\wedge E_{i_j}=0$. Setting $i_0:=\max \{i_j: 1\leq j\leq l\}$, since $\{E_i\}$ is decreasing, we thus have $E_{i_0}\wedge F=0$, contradicting that $0\notin \mathcal{F}$. Therefore, the desired triple $(\la,e_G,\la)$ in $F$ exists. Also, the fact $\iota_\la \Cap E$ for all $E\in \mathcal{F}$ implies $\iota_\la\in \mathcal{F}$ (because $\mathcal{F}$ is an ultrafilter, see Remark \ref{rem6.3}(1) below).

Furthermore, such $\iota_\la$ must be unique: if $\iota_\la,\iota_{\la'} \in \mathcal{F}$ with $(\la,e_G,\la), (\la',e_G,\la') \in F$, then $\iota_\la \iota_{\la'}\neq 0$ or equivalently $\Lam(\la,\la')\neq \emptyset$, and hence $\la=\la'$ by the ($\perp$) property for $F$.
\end{proof}

For the reader's convenience, we recall the definition of $\partial \La$ for a finitely aligned $k$-graph $\La$. We use the following special $k$-graphs.

\begin{ex}
Given any $m\in(\N\cup \{\infty\})^k$, let $\Omega_{k,m}$ be the category
$$\Omega_{k,m}:=\left\{(p,q)\in \N^k\times\N^k : p\leq q\leq m \right\}$$
with $r(p,q):=(p,p)$ and $s(p,q):=(q,q)$. Then $\Omega_{k,m}$ equipped with the degree map $d(p,q):=q-p$ is a $k$-graph.
\end{ex}

Following \cite[Definition 5.10]{far05}, a \emph{boundary path} in $\La$ is a graph morphism $x:\Omega_{k,m} \rightarrow \La$ such that for all $(p,p)\in\Omega_{k,m}^0$ and $X\in x(p,p)\mathbf{FE}(\La)$, there exists $\mu\in X$ with $x(p,p+d(\mu))=\mu$. We simply write $x(p,p)$ by $x(p)$. The set of all boundary paths in $\La$ is denoted by $\partial \La$. Then, for any $x\in \partial \La$, $r(x):=x(0)$ is defined as \emph{the range of $x$} and $d(x):=m\in (\N\cup \{\infty\})^k$ as \emph{the degree of $x$}. Recall from \cite[Lemma 5.13]{far05} that $v\partial\La=\{x\in \partial \La: r(x)=v\}$ is always nonempty for all $v\in \La$.

For every $x\in \partial\La$ and $n\leq d(x)$, \emph{the shift of $x$} is the boundary path $\sigma^n(x)\in\partial \La$ such that $d(\sigma^n(x))=d(x)-n$ and $\sigma^n(x)(p,q):=x(n+p,n+q)$ for $p\leq q\leq d(x)-n$. Notice that the factorisation property implies $x(0,n)\sigma^n(x)=x$.

Furthermore, in \cite[Proposition 5.4]{far05}, the authors introduced a locally compact Hausdorff topology on the path space $X_\La$ and then showed that $\partial \La$ is a closed subset of $X_\La$ \cite[Lemma 5.12]{far05}. In particular, we may consider $\partial \La$ as a topological space  generated by the compact open neighborhoods
$$Z(\la):=\la \partial \La=\{\la y: y\in \partial \La, r(y)=s(\la)\}.$$
We refer the reader to \cite{far05} for details.

The following descriptions for ultrafilters and tight filters are well-known and will be useful in the sequel (they may be derived from \cite[Theorem IV.3.12]{bur81} and \cite[Proposition 11.9]{exe08}, respectively). The argument after Proposition \ref{prop6.4} insures that the tight spectrum of $\sg$ coincides with the ultrafilter space.

\begin{rem}\label{rem6.3}
The ultrafilters and tight filters in $\es$ can be described as follows:
\begin{enumerate}[(1)]
  \item A filter $\mathcal{F}$ in $\es$ is an ultrafilter if and only if for every $E\in \es$
  $$E\Cap F \hspace{3mm} \mathrm{for ~ all~} F\in \mathcal{F} \hspace{5mm} \Longrightarrow \hspace{5mm} E\in \mathcal{F}.$$
  \item A filter $\mathcal{F}$ in $\es$ is tight if and only if
  $$F\in \mathcal{F} ~~~ \mathrm{and} ~~~ \mathcal{C} ~ \mathrm{is ~ a ~ finite ~ cover ~ for ~} F \hspace{3mm} \Longrightarrow \hspace{3mm} \mathcal{C}\cap \mathcal{F}\neq \emptyset.$$
\end{enumerate}
\end{rem}

\begin{prop}\label{prop6.4}
Let $(G,\La)$ be a self-similar $k$-graph. For any $x\in \partial\La$ define
$$\mathcal{F}_x=\{E\in \es: (x(0,n),e_G,x(0,n)) \in E ~~ \mathrm{for ~ some ~} n\leq d(x)\}.$$
Then the map $x\mapsto \mathcal{F}_x$ is a homeomorphism from $\partial\La$ onto the ultrafilter space $\widehat{\mathcal{E}}_\infty(\sg)$.
\end{prop}

\begin{proof}
First, note that $\mathcal{F}_x=\bigcup_{n\leq d(x)} \iota_{x(0,n)}\uparrow$, so $\mathcal{F}_x$ is a filter in $\es$ for every $x\in \partial \La$. To see that $\mathcal{F}_x$ is an ultrafilter, we can apply Remark \ref{rem6.3}(1). For, pick some $E\in \es$ satisfying $E\Cap F$ for all $F\in \mathcal{F}_x$. In particular, $E\wedge \iota_{x(0,n)}\neq 0$ for all $n\in \N^k$ with $n\leq d(x)$, and since $E$ is finite we must have $(x(0,n'),e_G,x(0,n'))\in E$ for some $n'\leq d(x)$. Hence $E\in \mathcal{F}_x$, concluding that $\mathcal{F}_x$ is an ultrafilter by Remark \ref{rem6.3}(1).

As injectivity of the map is clear, let us prove the surjectivity. In order to do this, we show that every ultrafilter in $\mathcal{E}(\sg)$ is of the form $\mathcal{F}_x$ with $x\in \partial\La$. So, fix $\mathcal{F}\in \widehat{\mathcal{E}}_\infty(\sg)$. Define $T=\{\la\in\La:\iota_\la\in \mathcal{F}\}$ and let $m:=\bigvee_{\la\in T} d(\la)$. Recall that for every $\iota_\la,\iota_{\la'}\in \mathcal{F}$ we have $\iota_\la\iota_{\la'}\neq 0$, or equivalently $\Lam(\la,\la')\neq \emptyset$. Then for each $n\in \N^k$ with $n\leq m$, there is a unique $\la\in T$ such that $d(\la)=n$. We may thus define a well-defined graph homomorphism $x:\Omega_{k,m}\rightarrow \La$, by $x(0,n)=\la$ for every $n\leq m$, where $\la\in T$ with $d(\la)=n$. Now, we need to check that $x$ is a boundary path in $\La$.

{\bf Claim I:} Letting $m=(m_1,\ldots,m_k)\in (\N\cup\{\infty\})^k$, if $m_i<\infty$, then $x(m_ie_i)\La^{e_i}=\emptyset$.

To prove Claim I, we assume $m_i<\infty$ for some $1\leq i\leq k$. Write $\mu=x(0,m_ie_i)$ for simplicity. Note that $\iota_\mu\wedge \iota_\la \neq 0$ for all $\la\in T$ because $\mathcal{F}$ is an ultrafilter. If, on the contrary, there exists $\alpha\in x(m_ie_i)\La^{e_i}=s(\mu)\La^{e_i}$, then we have $\iota_{\mu\alpha}\wedge \iota_\la \neq 0$ as well, or $\iota_{\mu\alpha}\Cap \iota_\la$ for all $\la\in T$. Thus, using Lemma \ref{lem6.1}, $\iota_{\mu\alpha}\Cap F$ for all $F\in \mathcal{F}$, which follows $\iota_{\mu\alpha}\in \mathcal{F}$ and $\mu\alpha\in T$ by Remark \ref{rem6.3}(1). But $m_ie_i<(m_i+1)e_i=d(\mu\alpha)$, contradicting the definition $m=\bigvee_{\la\in T}d(\la)$. Therefore, Claim (I) holds.

{\bf Claim II:} $x$ is a boundary path.

In order to prove Claim II, let $n\in \mathbb{N}$ with $n\leq m$ and assume $\Gamma$ is a finite exhaustive set in $x(n)\La$. Set $n':=\bigvee_{\gamma\in \Gamma}d(\gamma)$. Observe that Claim I implies $n+n'\leq m$. Since $x(n,n+n')\in x(n)\La$ and $\Gamma$ is exhaustive, there exists $\gamma\in \Gamma$ such that $\Lam(x(n,n+n'),\gamma)\neq \emptyset$. But $d(\gamma)\leq n'=d(x(n,n+n'))$, hence we must have $x(n,n+n')=\gamma\beta$ for some $\beta\in \La$. Consequently, $\gamma=x(n,n+d(\gamma))$, and $x$ is a boundary path.

Moreover, it is clear that $\mathcal{F}_x\subseteq \mathcal{F}$, and since $\mathcal{F}_x$ is an ultrafilter as saw in the first part of proof, we conclude $\mathcal{F}=\mathcal{F}_x$. This proves that the map is surjective.

It remains to show that the correspondence $x\mapsto \mathcal{F}_x$ is open and continuous. For this, recall that $\{Z(\la):\la\in \La\}$ is a compact open basis for the topology on $\partial \La$ \cite[Section 5]{far05}. Hence, a net $\{x_t\}_{t\in I}$ converges to $x$ in $\partial \La$ if and only if for each $\la\in \La$ there is $t_0\in I$ such that for every $t\geq t_0$,
$$x_t\in Z(\la) \hspace{3mm} \Longleftrightarrow \hspace{3mm} x\in Z(\la).$$
On the other hand, a net $\{\mathcal{F}_{x_t}\}_{t\in I}$ converges to $\mathcal{F}_x$ in $\widehat{\mathcal{E}}_{\infty}(\sg)$ if and only if for each $E\in \es$, there exists $t_0\in I$ such that for every $t\geq t_0$, we have
$$E\in \mathcal{F}_{x_t} \hspace{3mm} \Longleftrightarrow \hspace{3mm} E\in \mathcal{F}_x.$$
Moreover, using the fact $\mathcal{F}_x=\bigcup_{n\leq d(x)}  \iota_{x(0,n)}\uparrow$, we have
\begin{align*}
E\in \mathcal{F}_x \hspace{3mm} &\Longleftrightarrow \hspace{3mm} \exists n\leq d(x) ~\mathrm{such ~ that~} (x(0,n),e_G,x(0,n)\in E\\
&\Longleftrightarrow \hspace{3mm} x \mathrm{~is ~ of ~ the ~ form ~} x=\la y \mathrm{~for ~ some ~ } \la\in \La \mathrm{~ with ~} (\la,e_G,\la)\in E.
\end{align*}
Combining the above observations follows that the correspondence $x\mapsto \mathcal{F}_x$ is a homeomorphism. We are done.
\end{proof}

In light of Corollary \ref{cor3.5}, the idempotent set $\es$ of $\sg$ equals to that of the inverse semigroup $\mathcal{S}_{\La}$ introduced in \cite{far05} ($\mathcal{S}_{\La}$ is just $\mathcal{S}_{\{e_G\},\La}$ with the trivial group $\{e_G\}$). In particular, we have $\widehat{\mathcal{E}}_0(\sg)=\widehat{\mathcal{E}}_0(\mathcal{S}_\La)$ and $\widehat{\mathcal{E}}_{\mathrm{tight}}(\sg)=\widehat{\mathcal{E}}_{\mathrm{tight}}(\mathcal{S}_\La)$. So, $\widehat{\mathcal{E}}_0(\sg)$ is homeomorphic to path space $X_\La$ of \cite{far05} (see \cite[Remark 5.9]{far05}). On the other hand, since $\widehat{\mathcal{E}}_{\mathrm{tight}}(\sg)$ is the closure of $\widehat{\mathcal{E}}_\infty(\sg)$, Proposition \ref{prop6.4} and \cite[Lemma 5.12]{far05} imply that $\widehat{\mathcal{E}}_{\mathrm{tight}}(\sg)$ coincides with $\widehat{\mathcal{E}}_\infty(\sg)=\partial \La$; hence we may identify $\widehat{\mathcal{E}}_{\mathrm{tight}}(\sg)$ with $\partial \La$ via the correspondence of Proposition \ref{prop6.4}. For the reader's convenience, we restate here the construction of $\gt$ taking into account this identification.

For each $\mu\in \La$, let $Z(\mu):=\mu \partial \La=\{\mu x:x\in \partial \La, r(x)=s(\mu)\}$,
which is a subset of $\partial \La$. Given any $F\in \sg$, if
$$D^{F^*F}:=\bigcup_{(\mu,g,\nu)\in F} Z(\nu)  \hspace{5mm} \mathrm{and} \hspace{5mm} D^{FF^*}:=\bigcup_{(\mu,g,\nu)\in F}Z(\mu),$$
define $\theta_F:D^{F^*F}\rightarrow D^{FF^*}$ by $\theta_F(\nu x)=\mu(g \cdot x)$. Then the action $\theta:\sg \curvearrowright \partial \La$ is $F\mapsto \theta_F$, and its germs are of the form $(F,\nu x)$ such that $F\in \sg$, there is a triple $(\mu,g,\nu)$ in $F$, and $x\in Z(s(\nu))$. Thus the groupoid of germs is
\begin{equation}\label{eq6.1}
\gt=\left\{[F,\nu x]: F\in \sg, \exists(\mu,g,\nu)\in F, x\in Z(s(\nu)) \right\}
\end{equation}
where $[F,\nu x]=[F',\gamma y]$ if and only if $\nu x=\gamma y$ and there exists $E\in \es$ such $\theta_E(\nu x)=\nu x$ and $FE=F'E$. Moreover, the unite space is
$$\gt^{(0)}=\left\{[E,\mu x]: E\in \es, (\mu,e_G,\mu)\in E, x\in Z(s(\mu)) \right\}$$
which can be identified with $\partial \La$ via $[E,\mu x]\mapsto \mu x$. So, the source and range maps are
$$s([F,\nu x])=\nu x \hspace{5mm}  \mathrm{and}  \hspace{5mm} r([F,\nu x])=\mu(g \cdot x)$$
where there is a triple $(\mu,g,\nu)$ in $F$. Note that the property ($\perp$) for elements of $F$ insures that the definition of $r$ is well-defined.

Furthermore, the sets of the form
$$\Theta(F,Z(\nu))=\left\{ [F,\nu x]\in \gt: x\in Z(s(\nu)) \right\}$$
are compact open bisections generating the topology on $\gt$ (it follows from \cite[Proposition 5.7]{far05} that each $Z(\nu)$ is compact, and hence $\Theta(F,Z(\nu))$ is as well). Thus $\gt$ is an ample groupoid.

In the end of this section, we have a look at the relation between $\gt$ and the tight groupoid $\mathcal{G}_{\mathrm{tight}}(\mathcal{S}(\La \rtimes^{\varphi} G))$ of $\La \rtimes^{\varphi} G$ as an LCSC in \cite{ort20}. First, by comparing Corollary \ref{cor3.5} with \cite[Lemma 2.9]{ort20} we see that the idempotent semilattices of $\mathcal{S}_{\La}$, $\sg$ and $\mathcal{T}(\La \rtimes^{\varphi} G)$ coincide, so do their tight spectrums, being equal to $\partial \La$ by Proposition \ref{prop6.4}. Thus, since $\mathcal{G}_{\mathrm{tight}}(\mathcal{S}(\La \rtimes^{\varphi} G))= \mathcal{G}_{\mathrm{tight}}(\mathcal{T}(\La \rtimes^{\varphi} G))$, elements of $\mathcal{G}_{\mathrm{tight}}(\mathcal{S}(\La \rtimes^{\varphi} G))$  are of the form $[\tau^{(\mu,g)}\sigma^{(\nu,h)};\nu x]$ where $x\in Z(s(\nu))= s(\nu)\partial \La$. As one may also write
$$\tau^{(\mu,g)}\sigma^{(\nu,h)}=\tau^{(\mu,g)}\sigma^{(h^{-1}\cdot s(\nu),h^{-1})} \tau^{(s(\nu),h)}\sigma^{(\nu,e_G)}= \tau^{(\mu,gh^{-1})}\sigma^{(\nu,e_G)},$$
hence
$$\mathcal{G}_{\mathrm{tight}}(\mathcal{S}(\La \rtimes^{\varphi} G))=\{ [\tau^{(\mu,g)}\sigma^{(\nu,e_G)};\nu x]: s(\mu)=g \cdot s(\nu) ~~ \mathrm{and} ~~ x\in Z(s(\nu)) \}.$$

One the other side, for each $[F; \nu x]\in \gt$ in (\ref{eq6.1}) there is a unique triple $(\mu,g,\nu)$ in $F$ so that
$$F\{(\nu,e_G,\nu)\}=\{(\mu,g,\nu)\}\{(\nu,e_G,\nu)\}=\{(\mu,g,\nu)\}.$$
So $(F,\nu x) \sim (\{(\mu,g,\nu)\}; \nu x)$, and (\ref{eq6.1}) turns out
\begin{equation}\label{eq6.2}
\gt=\left\{[\{(\mu,g,\nu)\},\nu x]: s(\mu)=g \cdot s(\nu) ~~ \mathrm{and} ~~ x\in Z(s(\nu)) \right\}.
\end{equation}
Therefore, we obtain the following:

\begin{prop}\label{prop6.5}
Let $(G,\La,\varphi)$ be a self-similar $k$-graph and $\La \rtimes^{\varphi}G$ its Zappa-Sz$\acute{e}$p product. Then the map
$$[\{(\mu,g,\nu)\},\nu x] \longmapsto [\tau^{(\mu,g)}\sigma^{(\nu,e_G)};\nu x]$$
is a groupoid isomorphism from $\gt$ onto $\mathcal{G}_{\mathrm{tight}}(\mathcal{S}(\La \rtimes^{\varphi} G))$.
\end{prop}

\begin{proof}
It is straightforward to see that the map is a surjective groupoid homomorphism. For the injectivity, suppose that
$$(\tau^{(\mu,g)}\sigma^{(\nu,e_G)};\nu x)\sim (\tau^{(\alpha,h)}\sigma^{(\beta,e_G)};\beta y).$$
Then $\nu x=\beta y$ and there exists an idempotent $f=\tau^{(\gamma,e_G)}\sigma^{(\gamma,e_G)}\in \mathcal{E}(\mathcal{S}(\La \rtimes^{\varphi} G))$ such that
\begin{equation}\label{eq6.3}
\tau^{(\mu,g)}\sigma^{(\nu,e_G)} f= \tau^{(\alpha,h)}\sigma^{(\beta,e_G)} f.
\end{equation}
Note that in this case $\gamma$ must be an initial segment of $\nu x=\beta y$. Let $\lambda:=(\nu x)(0, d(\nu) \vee d(\beta) \vee d(\gamma)$. Then by (\ref{eq6.3}),
$$\tau^{(\mu,g)}\sigma^{(\nu,e_G)} f (\lambda)=\tau^{(\mu,g)}\sigma^{(\nu,e_G)} (\lambda) =\mu\big(g \cdot \lambda(d(\nu),d(\lambda))\big)$$
is equal to
$$\tau^{(\alpha,h)}\sigma^{(\beta,e_G)} f (\lambda)=\tau^{(\alpha,h)}\sigma^{(\beta,e_G)} (\lambda)= \alpha \big(h\cdot \lambda(d(\beta),d(\lambda)) \big).$$
This says that
$$\{(\mu,g,\nu)\} \{(\lambda,e_G,\lambda)\}=\{(\alpha,h,\beta)\} \{(\lambda,e_G,\lambda)\}$$
as elements of $\sg$, and since $\lambda$ is an initial segment of $\nu x=\beta y$, we obtain that
$$[\{(\mu,g,\nu)\};\nu x]= [\{(\alpha,h,\beta)\};\beta y]$$
in $\gt$. Consequently, the map is injective.

To complete the proof, we should only notice that the map induces a one-to-one correspondence between open basic sets $[\{(\mu,g,\nu)\};Z(\nu \gamma)]$ in $\gt$ and the ones $[\tau^{(\mu,g)}\sigma^{(\nu,e_G)};Z(\nu \gamma)]$ in $\mathcal{G}_{\mathrm{tight}}(\mathcal{S}(\La \rtimes^{\varphi} G))$, and therefore it is a homeomorphism.
\end{proof}

In particular, combining Proposition \ref{prop6.5} with \cite[Proposition 5.2]{ort20} implies that $\gt$ is isomorphic to the Spielberg's groupoid $G|_{\partial (\La\rtimes^{\varphi} G)}$ of \cite{spi12,spi20}.

\section{Linking to \cite{li21,li21-ideal}}\label{sec7}

In this section, we assume $(G,\La)$ is a pseudo free self-similar $k$-graph with a row-finite $\La$ without sources. Li and Yang in \cite{li21,li21-ideal} introduced an ample groupoid $\mathcal{G}_{G,\La}$ such that $\og\cong C^*(\mathcal{G}_{G,\La})$ provided $G$ is amenable. We are going to prove that our groupoid of germs $\gt$ is isomorphic to $\mathcal{G}_{G,\La}$ in this case. Let us first recall the definition of $\mathcal{G}_{G,\La}$ from \cite{li21,li21-ideal} for convenience.

Let $C(\N^k,G)$ be the group of maps $f:\N^k \rightarrow G$ with the pointwise multiplication. Given $f,g\in C(\N^k,G)$, define the equivalence relation $f\sim g$ whenever there exists $n_0\in \N^k$ such that $f(n)=g(n)$ for all $n\geq n_0$. Define $Q(\N^k,G):=C(\N^k,G)/\sim$. Also, for each $z\in\Z^k$, if $\mathcal{T}_z: C(\N^k,G)\rightarrow C(\N^k,G)$ is the automorphism defined by
$$\mathcal{T}_z(f)(n)=\left\{
    \begin{array}{ll}
      f(n-z) & n-z\geq 0 \\
      e_G & \mathrm{otherwise}
    \end{array}
  \right. \hspace{5mm} (f\in C(\N^k,G),~ n\in \N^k),
$$
then $\mathcal{T}_z$ induces an automorphism, denoted again by $\mathcal{T}_z$, on $Q(\N^k,G)$ such that $\mathcal{T}_z([f])=[\mathcal{T}_z(f)]$. So, $\mathcal{T}:\Z^k\rightarrow \mathrm{Aut}Q(\N^k,G)$ is a homomorphism and one may consider the semidirect product group $Q(\N^k,G)\rtimes_{\mathcal{T}} \Z^k$.

\begin{defn}
Let $(G,\La)$ be a row-finite source-free self-similar $k$-graph. One may define $\mathcal{G}_{G,\La}$ as the subgroupoid
$$\g:=\left\{\big(\mu(g\cdot x);\mathcal{T}_{d(\mu)}([\varphi(g,x)]),d(\mu)-d(\nu);\nu x\big) :g\in G, ~\mu,\nu\in \La,~ s(\mu)=g\cdot s(\nu) \right\}$$
of $\La^\infty\times \big(Q(\N^k,G)\rtimes_{\mathcal{T}} \Z^k\big) \times \La^\infty$ with the range and source maps
$$r(x;[f],n-m;y)=x  \hspace{5mm} \mathrm{and} \hspace{5mm} s(x;[f],n-m;y)=y.$$
If we set
$$Z(\mu,g,\nu):=\left\{\big(\mu(g\cdot x);\mathcal{T}_{d(\mu)}([\varphi(g,x)]),d(\mu)-d(\nu);\nu x\big) :x\in s(\nu)\La^\infty \right\},$$
then $\mathcal{G}_{G,\La}$ endowed with the topology generated by
$$\mathcal{B}_{G,\La}:=\left\{Z(\mu,g,\nu):\mu,\nu\in \La,g\in G,s(\mu)=g\cdot s(\nu)\right\}$$
is a topological groupoid. In case $(G,\La)$ is pseudo free, \cite[Proposition 3.11]{li21-ideal} proves that $\g$ is a Hausdorff ample groupoid with the compact open basis $\mathcal{B}_{G,\La}$.
\end{defn}

\begin{prop}\label{prop7.2}
Let $(G,\La)$ be a self-similar $k$-graph over a row-finite source-free $k$-graph $\La$. Suppose moreover that $(G,\La)$ is pseudo free. Then $\psi:\gt \rightarrow \mathcal{G}_{G,\La}$ defined by
\begin{equation}\label{eq7.1}
[F,\nu x]\longmapsto \big(\mu(g\cdot x);\mathcal{T}_{d(\mu)}([\varphi(g,x)]), d(\mu)-d(\nu);\nu x \big),
\end{equation}
where $(\mu,g,\nu)$ is the unique triple in $F$ with $\nu$ as the third part, is an isomorphism from $\gt$ onto $\mathcal{G}_{G,\La}$ as topological groupoids.
\end{prop}

\begin{proof}
It is straightforward to check that $\psi$ is a surjective groupoid homomorphism. So, given $[F_1,\nu_1 x_1], [F_2,\nu_2 x_2]\in \gt$, we need to prove
\begin{align}\label{eq7.2}
\psi([F_1,\nu_1 x_1])=\psi([F_2,\nu_2 x_2]) ~~ & \Longleftrightarrow ~~ [F_1,\nu_1 x_1]=[F_2,\nu_2 x_2]\\
& \Longleftrightarrow ~~ (F_1,\nu_1 x_1) \sim (F_2,\nu_2 x_2) ~~ \mathrm{in} ~ \sg* \La^\infty. \nonumber
\end{align}

In order to do this, we first assume $\psi([F_1,\nu_1 x_1])=\psi([F_2,\nu_2 x_2])$. By (\ref{eq6.2}), we may assume that $F_1, F_2$ are singletons of the forms $F_1=\{(\mu_1,g_1,\nu_1)\}$ and $F_2=\{(\mu_2,g_2,\nu_2)\}$. Then by (\ref{eq7.1}) the equality $\psi([F_1,\nu_1 x_1])=\psi([F_2,\nu_2 x_2])$ means
\begin{enumerate}[(i)]
  \item $\mu_1(g_1\cdot x_1)=\mu_2 (g_2\cdot x_2)$,
  \item $\nu_1 x_1=\nu_2 x_2$,
  \item $d(\mu_1)-d(\nu_1)=d(\mu_2)-d(\nu_2)$, and
  \item $\mathcal{T}_{d(\mu_1)}([\varphi(g_1,x_1)])=\mathcal{T}_{d(\mu_2)}([\varphi(g_2,x_2)])$.
\end{enumerate}
Fix sufficiently large $n\in \mathbb{N}^k$. If define $n^j:=n- d(\nu_j)$ and $\la_j:=x_j(0,n^j)$ for $j=1,2$, then (ii) implies
\begin{equation}\label{eq7.3}
\nu_1 \la_1=\nu_2 \la_2.
\end{equation}
We also claim that $\mu_1(g_1\cdot \la_1)=\mu_2(g_2\cdot \la_2)$. To see this claim, using (i) above, it suffices to have $d(\mu_1(g_1\cdot \la_1))=d(\mu_2(g_2\cdot \la_2))$, or equivalently
$$d(\mu_1)+d(\la_1)=d(\mu_2)+d(\la_2).$$
This follows simply from (iii) and that
\begin{align*}
d(\la_1)-d(\la_2)&=d(\nu_2)-d(\nu_1)  \hspace{10mm} (\mathrm{by}~ (\ref{eq7.3}))\\
&=d(\mu_2)-d(\mu_1) \hspace{10mm} (\mathrm{by}~ \mathrm{(iii)}).
\end{align*}

Furthermore, (iv) says that there exists $n_0\in \mathbb{N}^k$ such that
$$\varphi(g_1,x_1)(n-d(\mu_1))=\varphi(g_2,x_2)(n-d(\mu_2)) \hspace{10mm} (\forall n\geq n_0),$$
which means
$$\varphi(g_1,x_1(0,n-d(\mu_1))=\varphi(g_2,x_2(0,n-d(\mu_2))$$
for all $n\geq n_0$. In particular, since $d(\la_1)-d(\la_2)=d(\mu_2)-d(\mu_1)$, it follows $\varphi(g_1,\la_1)=\varphi(g_2,\la_2)$. Now define $\nu:=\nu_1\la_1=\nu_2 \la_2$, $\mu:=\mu_1(g_1\cdot \la_1)=\mu_2(g_2\cdot \la_2)$, $g:=\varphi(g_1,\la_1)=\varphi(g_2,\la_2)$ and $x:=\sigma^{d(\la_1)}(x_1)=\sigma^{d(\la_2)}(x_2)$. Then
\begin{align*}
\{(\mu_1,g_1,\nu_1)\}\{(\nu,e_G,\nu)\}& =\left\{\big(\mu_1(g_1\cdot \la_1),\varphi(g_1,\la_1),\nu\big)\right\}  \hspace{5mm} (\mathrm{by ~ Definition ~} \ref{defn3.1})\\
&=\{(\mu,g,\nu)\}\\
&=\{(\mu,g,\nu)\}\{(\nu,e_G,\nu)\},
\end{align*}
and hence
$$\big(\{(\mu_1,g_1,\nu_1)\},\nu_1 x_1\big) \sim \big(\{(\mu,g,\nu)\},\nu x\big)$$
in $\sg*\La^\infty$. Similarly, we can show $\big(\{(\mu_2,g_2,\nu_2)\},\nu_2 x_2\big) \sim \big(\{(\mu,g,\nu)\},\nu x\big)$ as well, and therefore $\big[\{(\mu_1,g_1,\nu_1)\},\nu_1 x_1\big] = \big[\{(\mu_2,g_2,\nu_2)\},\nu_2 x_2\big]$ proving the ``if" part of (\ref{eq7.2}). The converse implication of (\ref{eq7.2}) is analogous and left to the reader.

It remains to show that $\psi$ is an open and continuous map. Recall that the open sets $[\{(\mu,e_G,\nu)\}, Z(\nu)]$ form a basis for the topology on $\gt$. Since
\begin{align*}
\psi([\{(\mu,g,\nu)\},Z(\nu)])&=\left\{\big(\mu (g\cdot x);\mathcal{T}_{d(\mu)}([\varphi(g,x)]), d(\mu)-d(\nu); \nu x\big): ~ x\in s(\nu)\La^\infty \right\}\\
&=Z(\mu,g,\nu),
\end{align*}
we see that $\psi$ maps precisely the basic open sets in $\gt$ onto the basic open sets in $\mathcal{G}_{G,\La}$, and therefore $\psi$ is a homeomorphism. This completes the proof.
\end{proof}


\section{Minimality of $\gt$}\label{sec8}

In this section, we show that the minimality of $\gt$ is completely related to the notion of $G$-cofinality for $\La$.

\begin{defn}
Let $\mathcal{G}$ be a groupoid. A subset $U\subseteq \mathcal{G}^{(0)}$ is called \emph{invariant} if for every $\alpha\in \mathcal{G}$ we have $s(\alpha)\in U ~~\Longleftrightarrow ~~ r(\alpha)\in U$. We say $\mathcal{G}$ is a \emph{minimal groupoid} if the only invariant open subsets of $\mathcal{G}^{(0)}$ are the empty set and $\mathcal{G}^{(0)}$ itself.
\end{defn}

\begin{defn}\label{defn8.2}
Let $(G,\La)$ be a self-similar $k$-graph. We say that $(G,\La)$ is \emph{$G$-cofinal} if for every $v\in \La^0$ and $x\in \partial\La$, there are $n\leq d(x)$ and $g\in G$ such that $v\La(g\cdot x(n))\neq \emptyset$.
\end{defn}

The main result of this section is:

\begin{thm}\label{thm8.3}
Let $(G,\La)$ be a self-similar $k$-graph. Then the following are equivalent.
\begin{enumerate}[(1)]
  \item The action $\theta:\sg \curvearrowright \partial\La$ is irreducible (in the sense of \cite[Definition 5.1(3)]{exe16}).
  \item The groupoid $\gt$ is minimal.
  \item $(G,\La)$ is $G$-cofinal.
\end{enumerate}
\end{thm}

To prove the above theorem, we only need to get the following proposition.

\begin{prop}\label{prop8.4}
A self-similar $k$-graph $(G,\La)$ is $G$-cofinal if and only if for every $E,F\in \es$, there are finitely many elements $F_1,\ldots, F_l\in \sg$ such that $\{F_i EF_i^*\}_{1\leq i\leq l}$ is an outer cover for $F$ (Definition \ref{defn4.1}).
\end{prop}

\begin{proof}
Let us first assume $(G,\La)$ is $G$-cofinal and prove the ``only if" part. Note that for $(\mu,g,\nu),(\xi,h,\eta)\in F_i$ and $\iota_\gamma \in E$, if
$$\{(\mu,g,\nu)\}\iota_\gamma \{(\xi,h,\eta)\}^*\neq 0,$$
then we must have $\Lam(\nu,\gamma)\neq \emptyset$ and $\Lam(\gamma,\eta)\neq \emptyset$. Hence $\Lam(\nu,\eta)\neq \emptyset$, which yields $(\mu,g,\nu)=(\xi,h,\eta)$ by the ($\perp$) property. Thus
\begin{align*}
F_iEF_i&= \bigcup_{(\mu,g,\nu)\in F_i}\{(\mu,g,\nu)\} ~ E ~ \{(\mu,g,\nu)\}^*\\
&= \bigcup_{(\mu,g,\nu)\in F_i} \bigcup_{(\la,e_G,\la)\in E} \{(\mu,g,\nu)\} \iota_\la \{(\nu,g^{-1},\mu)\}.
\end{align*}
Moreover, a set $\mathcal{C}\subseteq \es$ is an outer cover for $F$ if and only if it is an outer cover for each singleton $\iota_\gamma$ with $(\gamma,e_G,\gamma)\in F$. Therefore, without loss of generality, it suffices to prove the statement only for singletons $E=\iota_\la$ and $F=\iota_\gamma$ in $\es$ and try to find $\{(\mu_i,g_i,\nu_i)\}_{i=1}^l \subseteq \sg$ such that $\{\{(\mu_i,g_i,\nu_i)\} \iota_\la \{(\mu_i,g_i,\nu_i)\}^*\}_{1\leq i\leq l}$ is an outer cover for $\iota_\gamma$.

So fix $\iota_\la,\iota_\gamma\in \es$. For each $x\in Z(\gamma)$, the $G$-cofinality says that there is an $n_x\leq d(x)$ such that $s(\la)\La (G\cdot x(n_x))\neq \emptyset$. Here, we may enlarge $n_x$ sufficiently to have $d(\gamma)\leq n_x$. So, since $Z(\gamma)$ is compact and covered as
$$Z(\gamma)\subseteq \bigcup_{x\in Z(\gamma)} Z(x(0,n_x))$$
by open cylinders $Z(x(0,n_x))$, there exists a finite set $B\subseteq Z(\gamma)$ such that
\begin{equation}\label{eq8.1}
Z(\gamma)\subseteq \bigcup_{x\in B} Z(x(0,n_x)).
\end{equation}

We claim that $\{\iota_{x(0,n_x)}\}_{x\in B}$ is a cover for $\iota_\gamma$. To see this, note that for any idempotent $\iota_\eta$ with $\iota_\eta\leq \iota_\gamma$ we have $\iota_\eta \iota_\gamma=\iota_\eta$, so $\eta=\gamma \tau$ for some $\tau\in \La$. Hence, we may pick  $y\in Z(\gamma)$ such that $y(0,d(\eta))=\eta$  (for example $y:=\eta z$ for some $z\in s(\eta)\partial \La$). By (\ref{eq8.1}), there exists $x\in B$ such that $y\in Z(x(0,n_x))$. Thus, we have $x(0,n_x)=y(0,n_x)$, and $\eta=y(0,d(\eta))$ implies $\Lam(x(0,n_x),\eta)\neq \emptyset$. Consequently, $\iota_{x(0,n_x)}\Cap \iota_\eta$, and the claim holds.

Furthermore, for every $x\in B$, the $G$-cofinal property gives $\tau_x \in s(\la) \La (G \cdot x(n_x))$ and $g_x\in G$ such that $s(\tau_x)=g_x\cdot x(n_x)$. If define
$$s_x:=(x(0,n_x),g_x,\la\tau_x)\in \sg \hspace{5mm} (\forall x\in B),$$
then
$$s_x\iota_\la s_x^*=\iota_{x(0,n_x)},$$
and the above claim concludes that $\{s_x \iota_\la s_x^*\}_{x\in B}$ is a finite cover for $\iota_\gamma$ as desired.

To prove the ``if" part, we fix $x\in \partial\La$ and $v\in \La^0$. By hypothesis, there are $F_1,\ldots,F_l\in \sg$ such that $\{F_i\iota_v F_i^*\}_{1\leq i\leq l}$ is an outer cover for $\iota_{r(x)}$. Observe again that for $(\mu,g,\nu),(\xi,h,\eta)\in F_i$, if
$$\{(\mu,g,\nu)\} \iota_v \{(\xi,h,\eta)\}^*=\{(\mu,g,\nu)\}\{(\eta,h^{-1},\xi)\}\neq 0,$$
then (\ref{eq3.1}) implies $\Lam(\nu,\eta)\neq \emptyset$, so $(\mu,g,\nu)=(\xi,h,\eta)$ by the ($\perp$) property. Also, for each $(\mu,g,\nu)\in F_i$, one can compute
\begin{align}\label{eq8.2}
\{(\mu,g,\nu)\} \iota_v \{(\mu,g,\nu)\}^*&=\{(\mu,g,\nu)\}\{ (v,e_G,v)\}\{(\nu,g^{-1},\mu)\}\\
&=\left\{
 \begin{array}{ll}
 \iota_\mu & \mathrm{if} ~ r(\nu)=v \\
 0 & \mathrm{otherwise}.
 \end{array}
 \right. \nonumber
\end{align}
Writing $\bigcup_{i=1}^l F_i=\{(\mu_j,g_j,\nu_j): 1\leq j\leq l'\}$ and $s_j:=\{(\mu_j,g_j,\nu_j)\}$ for $1\leq j\leq l'$, it follows that $\{s_j \iota_v s_j^*\}_{1\leq j\leq l'}$ is an outer cover for $\iota_{r(x)}$.

We now want to find a path $\mu_{j_0}$ in $\{\mu_j: 1\leq j\leq l'\}$ such that $x\in \mu_{j_0} \partial\La$. To do this, let $n:=\bigvee_{j=1}^{l'} d(\mu_j)$. Since $\iota_{x(0,n)}\leq \iota_{r(x)}$, there exists an element $s_{j_0} \iota_v s_{j_0}^*$ $(=\iota_{\mu_{j_0}})$ such that $\iota_{\mu_{j_0}}\Cap \iota_{x(0,n)}$. Thus $\Lam(\mu_{j_0},x(0,n))\neq \emptyset$, and we deduce $x(0,n)=\mu_{j_0}\tau$ for some $\tau\in \La$ because $d(\mu_{j_0})\leq n=d(x(0,n))$. Consequently, $x\in \mu_{j_0} \partial \La$ as desired.

On the other hand, by (\ref{eq8.2}), we have $r(\nu_{j_0})=v$ and
$$s(\nu_{j_0})=g_{j_0}^{-1}\cdot s(\mu_{j_0})=g_{j_0}^{-1}\cdot x(d(\mu_{j_0})),$$
meaning that $\nu_{j_0}\in v\La (g_{j_0}^{-1}\cdot x(d(\mu_{j_0})))$. This concludes that $(G,\La)$ is $G$-cofinal.
\end{proof}

\begin{proof}[Proof of Theorem \ref{thm8.3}.]
The statement follows from Proposition \ref{prop8.4} and \cite[Theorem 5.5]{exe16}.
\end{proof}


\section{Effective property of $\gt$}\label{sec9}

The notion of topological freeness is a generalization of that for group actions, which is studied in \cite[Section 4]{exe16}. It is proved in \cite[Theorem 4.7]{exe16} that an inverse semigroup action $S\curvearrowright X$ is topologically free if and only if the corresponding groupoid of germs is effective. Also, for the standard action $\mathcal{S}_{G,E}\curvearrowright E^\infty$ associated to a self-similar graph $(G,E)$, there is a characterization of this result in \cite[Theorem 14.10]{exe17} via the properties of $(G,E)$. In this section we give an aperiodicity condition for $(G,\La)$ under which the groupoid $\gt$ is effective (or equivalently, the action $\theta:\sg\curvearrowright \partial \La$ is topological free).

Recall that a topological groupoid $\mathcal{G}$ is said to be \emph{effective}\footnote{In the literature, the term \emph{essentially principal} is also used for this property (in \cite{ren80,exe16} for example).} if the interior of isotropy group bundle $\{\alpha\in \mathcal{G}:s(\alpha)=r(\alpha)\}$ in $\mathcal{G}$ is just $\mathcal{G}^{(0)}$. According to \cite[Proposition 3.1]{ren08}, in case $\mathcal{G}$ is second countable and Hausdorff, and $\mathcal{G}^{(0)}$ has the Baire property, then $\mathcal{G}$ is effective if and only if the units $x$ in $\mathcal{G}^{(0)}$ with trivial isotropy (meaning that $\{\alpha\in \mathcal{G}:r(\alpha)=s(\alpha)=x\}=\{x\}$) is dense in $\mathcal{G}^{(0)}$.

Furthermore, the adoption of \cite[Definition 4.1]{exe16} for the action $\theta:\sg \curvearrowright \partial \La$ is the following:

\begin{defn}
Let $(G,\La)$ be a self-similar $k$-graph and consider the action $\theta:\sg \curvearrowright \partial \La$ defined in the end of Section 6.
\begin{enumerate}[(1)]
\item If $F\in \sg$ and $x\in D^{F^* F}$ such that $\theta_F(x)=x$, we say $x$ is a \emph{fixed point for $F$}.
\item In case there is $E\in \es$ with $E\leq F$ and $x\in D^E$, $x$ is called a \emph{trivial fixed point} for $F$.
\item We say that $\theta$ is a \emph{topologically free} action if for every $F\in \sg$, the interior of the fixed point set for $F$ contains only trivial fixed points.
\end{enumerate}
\end{defn}

We now define an aperiodic condition for self-similar $k$-graphs $(G,\La)$, which implies the topological freeness of $\theta:\sg\curvearrowright \partial \La$ (and hence $\gt$ is effective). Note that our condition generalizes Condition (A) of \cite[Theorem 7.1]{far05} to self-similar $k$-graphs, which is weaker than that in \cite[Definition 6.4]{li21}.

\begin{defn}[An aperiodicity condition]\label{defn9.2}
Let $(G,\La)$ be a self-similar $k$-graph as in Definition \ref{defn2.3}. We say $(G,\La)$ satisfies \emph{G-aperiodic condition (A)} if for every $v\in \La^0$, there exists $x\in v\partial \La$ such that $\sigma^m(x)=g\cdot\sigma^n(x)$ implies $m=n$ for all $m,n\leq d(x)$ and $g\in G$.
\end{defn}

Let us prove some lemmas before the main result of this section, Theorem \ref{thm9.6}.

\begin{lem}\label{lem9.4}
Let $g\in G$ and $x\in \partial\La$. Then
$$\sigma^n(g\cdot x)=\varphi(g,x(0,n))\cdot \sigma^n (x)$$
for every $n\leq d(x)$.
\end{lem}

\begin{proof}
This follows immediately from the identification
$$g\cdot x=(g\cdot x)(0,n)[(g\cdot x)(n,d(x))]=g\cdot (x(0,n))[\varphi(g,x(0,n))\cdot x(n,d(x))]$$
and the unique factorization of $\La$.
\end{proof}

\begin{lem}\label{lem9.5}
Let $(G,\La)$ satisfy $G$-aperiodic condition (A). If $\theta_s$ has an interior fixed point for $s=\{(\mu,g,\nu)\}\in \sg$, then $\nu=\mu$.
\end{lem}

\begin{proof}
Let $x\in \partial\La$ be an interior fixed point for $\theta_s$, necessarily of the form $x=\nu y$ with $y\in s(\nu)\partial\La$. Since the cylinder set $\{Z(\beta):\beta \in \La\}$ is a basis for the topology on $\partial\La$, there exists $\la\in s(\nu)\La$ such that $x\in Z(\nu\la)$ and boundary paths in $Z(\nu\la)=\nu\la Z(s(\la))$ are all fixed by $\theta_s$. Thus
\begin{align}\label{eq9.1}
\theta_s(y)&=y  \hspace{9.5mm} (\forall y\in Z(\nu\la) \nonumber \\
\Longrightarrow \hspace{5mm}  \theta_{\{(\mu,g,\nu)\}}(\nu\la z)&=\nu\la z \hspace{5mm}  (\forall z\in Z(s(\la))) \nonumber\\
\Longrightarrow  \hspace{11.25mm}  \mu(g\cdot (\la z))&=\nu \la z  \hspace{5mm}  (\forall z\in Z(s(\la)))\\
\Longrightarrow \hspace{11.2mm} \sigma^{d(\mu)}(\nu\la z)&= g\cdot (\la z)= g\cdot\sigma^{d(\nu)}(\nu\la z) \hspace{5mm}  (\forall z\in Z(s(\la))).\nonumber
\end{align}
So, since $d(\la)=d(g\cdot\la)$, Lemma \ref{lem9.4} implies that
$$\sigma^{d(\mu)}(z)=\varphi(g,\gamma)\cdot \sigma^{d(\nu)}(z) \hspace{5mm}  (\forall z\in Z(s(\la)))$$
where $\gamma:=(\la z)(0,d(\nu\la))$. Then, the $G$-aperiodic condition concludes $d(\mu)=d(\nu)$, and therefore $\nu=\mu$ by (\ref{eq9.1}).
\end{proof}

\begin{lem}\label{lem9.6}
Let $(G,\La)$ be a self-similar $k$-graph satisfying the following property: for $v\in \La^0$ and $g\in G$,
\begin{center}
``if $g\cdot x=x$ for all $x\in v\partial \La$, then there is $X\in v\mathbf{FE}(\La)$ such that \\
 every $\tau\in X$ is strongly fixed by $g$."
\end{center}
Let $s=\{(\mu,g,\mu)\}$ be an element in $\sg$. If $x\in Z(\mu)$ is an interior fixed point for $\theta_s$, then there exists $\la\in s(\mu)\La$ such that $x\in Z(\mu\la)$, $g\cdot\la=\la$, and $\varphi(g,\la)=e_G$.
\end{lem}

\begin{proof}
Since $x$ is an interior fixed point for $\theta_s$, there is some $\gamma\in s(\mu)\La$ such that $x\in Z(\mu\gamma)\subseteq Z(\mu)$ and all boundary paths in $Z(\mu\gamma)$ are fixed by $\theta_s$. So we can write
\begin{align*}
 \theta_s(\mu\gamma y)&=\mu\gamma y\\
\Longrightarrow \hspace{17mm} \mu(g\cdot (\gamma y)) &=\mu \gamma y\\
\Longrightarrow \hspace{7mm}  (g\cdot\gamma)(\varphi(g,\gamma)\cdot y)&=\gamma y
\end{align*}
for all $y\in Z(s(\gamma))$. Hence, $g\cdot\gamma=\gamma$ and $\varphi(g,\gamma)\cdot y=y$ for every $y\in Z(s(\gamma))=s(\gamma)\partial\La$ by the unique factorization property. Denote $h=\varphi(g,\gamma)$ for convenience. Now, hypothesis gives a finite exhaustive set $X\in s(\gamma)\mathbf{FE}(\La)$ such that every $\tau\in X$ is strongly fixed by $h$. Pick some $\tau\in X$ and define $\la:=\gamma \tau$. Then
$$g\cdot\la=(g\cdot\gamma)(\varphi(g,\gamma)\cdot\tau)=\gamma(h\cdot\tau)=\gamma \tau=\la$$
and further,
$$\varphi(g,\la)=\varphi(g,\gamma\tau)=\varphi(\varphi(g,\gamma),\tau)=\varphi(h,\tau)=e_G.$$
Consequently, $\la \in s(\mu)\La$ is the desired element of $\La$.
\end{proof}

Now we prove the main result of this section.

\begin{thm}\label{thm9.6}
Let $(G,\La)$ be a self-similar $k$-graph. Consider the following statements:
\begin{enumerate}[(1)]
  \item The action $\theta:\sg \curvearrowright \partial\La$ is topologically free.
  \item The groupoid $\gt$ is effective.
  \item \begin{itemize}
          \item[(a)] $(G,\La)$ satisfies $G$-aperiodic condition (A), and
          \item[(b)] for every $v\in \La^0$ and $g\in G$, we have the following implication:\\
          $g\cdot x=x ~ (\forall x\in v\partial\La) ~ \Longrightarrow ~ \exists X\in v\mathbf{FE}(\La)$ such that all $\tau\in X$ are strongly fixed by $g$.
        \end{itemize}
\end{enumerate}
Then (3) implies (1) and (2). Moreover, if $\mathbf{SF}_g$ is locally exhausted for every $g\in G\setminus \{e_G\}$ (see Definition \ref{defn5.1}), then all (1), (2) and (3) are equivalent.
\end{thm}

\begin{proof}
First, note that (1) $\Longleftrightarrow$ (2) is just \cite[Theorem 4.7]{exe16}, so we will prove (3) $\Longrightarrow$ (1).

(3) $\Longrightarrow$ (1). Assume $(G,\La)$ satisfies (a)+(b) in (3). In order to prove (1), we fix some $F \in \sg$ and show that every interior fixed point of $\theta_F$ is a trivial one. So, let $x\in \partial\La$ be an interior fixed point for $\theta_F$. Then there is a singleton $s=\{(\mu,g,\nu)\}\subseteq F$ such that $x=\nu y\in Z(\nu)$ and $\mu(g\cdot y)=\nu y$.  First, Lemma \ref{lem9.5} implies $\nu=\mu$. Also, by Lemma \ref{lem9.6}, there is $\la\in s(\mu)\La$ such that $g\cdot\la=\la$, $\varphi(g,\la)=e_G$, and $x=\la y$ for some $y\in s(\la)\partial\La$. Then the fact
\begin{align*}
s \iota_{\mu\la} & =\{(\mu,g,\mu)\}\{(\mu\la,e_G,\mu\la)\}\\
& =\{(\mu(g\cdot\la),\varphi(g,\la),\mu\la)\}\\
&= \{(\mu\la,e_G,\mu\la)\}\\
& =\iota_{\mu\la}
\end{align*}
says that $\iota_{\mu\la}\leq s$, and since $\theta_{\iota_{\mu\la}}(x)=\theta_{\iota_{\mu\la}}(\mu\la y)=x$, $x$ would be a trivial fixed point. Consequently, the action $\theta:\sg \curvearrowright \partial \La$ is topologically free.

For the second statement, we assume that $\mathbf{SF}_g$ is locally exhausted for any $g\in G\setminus\{e_G\}$, and prove (1), (2) $\Longrightarrow$ (3).

(2) $\Longrightarrow$ (3)(a). Since $\g$ is second countable and Hausdorff by Theorem \ref{thm5.2}, \cite[Proposition 3.1]{ren08} implies that $\gt$ is effective if and only if the set
\begin{equation}\label{eq9.2}
\{x\in \partial \La: \mathrm{Iso}(x)=\{x\}\}
\end{equation}
is dense in $\partial \La=\mathcal{G}_{\mathrm{tight}}^{(0)}(\sg)$, where $\mathrm{Iso}(x):=\{[F,x]\in \gt: s([F,x])=r([F,x])=x\}$. By way of contradiction, assume $(G,\La)$ does not satisfy $G$-aperiodicity condition (A). Then there is $v\in \La^0$ such that for every $x\in Z(v)$ we have $\sigma^{m}(x)=g \cdot \sigma^{n}(x)$ for some $g\in G$ and $m\neq n\leq d(x)$. Given $x\in Z(v)$, define $\mu:=x(0,m)$, $\nu:=x(0,n)$ and $F=\{(\mu,g,\nu)\}$. Then
$$\theta_F(x)=Fx= \mu (g \cdot \sigma^{n}(x))=\mu \sigma^{m}(x)=x$$
and hence $s([F,x])=x=r([F,x])$, which means $[F,x]\in \mathrm{Iso}(x)$. Note that because $d(\mu)\neq d(\nu)$, Lemma \ref{lem4.3} says that $F$ does not dominate any idempotent in $\sg$, so $[F,x]\neq x$\footnote{By identification $\partial \La=\mathcal{G}^{(0)}_{\mathrm{tight}}(\sg)$, $[F,x]=x$ means that $[F,x]=[E,x]$ for some idempotent $E\leq F$.}. It follows
$$Z(v) \cap \{x\in \partial \La: \mathrm{Iso}(x)=\{x\}\}=\emptyset,$$
so the set (\ref{eq9.2}) is not dense in $\partial \La$. Therefore, $\gt$ is not effective by \cite[Proposition 3.1]{ren08}.

It remains to prove (1) $\Longrightarrow$(3)(b). Assume $v\in \La^0$ and $g\in G$ such that $g\cdot x=x$ for every $x\in v\partial\La$. In particular, $g\cdot r(x)=r(g\cdot x)=r(x)$, so $s=\{(v,g,v)\}\in \sg$. Then every $x\in v\partial \La$ is fixed by $\theta_s$. Since $Z(v)=v\partial\La$ is open in $\partial\La$, every such $x$ is an interior fixed point for $\theta_s$, and the topological freeness implies that there is $\mu_x\in v\La$ such that $\iota_{\mu_x}\leq s$ and $\theta_{\iota_{\mu_x}}(x)=x$. Note that $\mu_x$ must be of the form $\mu_x=x(0,n)$ for some $n\leq d(x)$, and we have
$$\mu_x(g\cdot y)=\mu_x y$$
where $y=\sigma^n(x)$. Moreover, $\iota_{\mu_x}\leq s$ implies
\begin{align*}
\{(v,g,v)\}\{(\mu_x,e_G,\mu_x)\}&=\{(\mu_x,e_G,\mu_x)\}\\
\Longrightarrow \hspace{10mm} (\mu_x,\varphi(g,\mu_x),\mu_x)&=(\mu_x,e_G,\mu_x)\\
\Longrightarrow \hspace{24mm} \varphi(g,\mu_x)&=e_G.
\end{align*}
Therefore, $\mu_x$ is strongly fixed by $g$.

On the other hand, since
$$Z(v)=v\partial\La=\bigcup_{x\in v\partial\La}Z(\mu_x)$$
and $Z(v)$ is compact, there exists a finite set $X\subseteq \{\mu_x:x\in v\partial\La\}$ such that $Z(v)=\bigcup_{\mu\in X}Z(\mu)$. One may easily check that the the set $X\subseteq v\La$ is exhaustive, whence statement (3)(b) is proved.
\end{proof}

Note that, in case $(G,\La)$ is pseudo free, then $\mathbf{SF}_g=\emptyset$ for all $g\in G\setminus\{e_G\}$, and therefore statements (1), (2), and (3) in Theorem \ref{thm9.6} are equivalent.

\begin{rem}
If a self-similar $k$-graph $(G,\La)$ (with finitely aligned $\La$) is $G$-aperiodic in the sense of \cite[Definition 6.4]{li21} then it satisfies both conditions (a) and (b) in Theorem \ref{thm9.6}(3) above, but the converse does not necessarily hold for general self-similar $k$-graphs $(G,\La)$. Recall that in case $(G,\La)$ is pseudo free, then for any $g\in G\setminus \{e_G\}$ there are no strongly fixed paths by $g$. This follows in particular that $G$-aperiodicity of \cite[Definition 6.4]{li21} is equivalent to conditions (a)+(b) in Theorem \ref{thm9.6} provided $(G,\La)$ is pseudo free.
\end{rem}


\section{Simplicity of $\og$: the Hausdorff case}\label{sec10}

Given a Hausdorff locally compact and second-countable $\acute{\mathrm{e}}$tale groupoid $\mathcal{G}$, \cite[Theorem 5.1]{bro14} characterizes the simplicity of full $C^*$-algebra $C^*(\mathcal{G})$. Hence, by applying Theorems \ref{thm5.2}, \ref{thm8.3} and \ref{thm9.6}, we obtain immediately the following proposition.

\begin{prop}\label{prop10.1}
Let $(G,\La)$ be a self-similar $k$-graph such that $\mathbf{SF}_g$ is locally exhausted for every $g\in G\setminus \{e_G\}$ (Definition \ref{defn5.1}). If
\begin{enumerate}[(1)]
\item $(G,\La)$ is $G$-cofinal (Definition \ref{defn8.2}),
\item $(G,\La)$ satisfies $G$-aperiodic condition (A), and
\item for each $v\in \La^0$ and $g\in G$ satisfying the property $g\cdot x=x ~~ (\forall x\in v\partial\La)$, there exists $X\in v\mathbf{FE}(\La)$ such that each $\tau\in X$ is strongly fixed by $g$,
\end{enumerate}
then $C^*_{\mathrm{red}}(\sg)=C^*_{\mathrm{red}}(\gt)$ is a simple $C^*$-algebra. If furthermore $\gt$ is amenable in the sense of \cite{ana00}, then $\og$ is simple.
\end{prop}

\begin{proof}
Recall that $\gt$ is a second-countable ample groupoid, which is Hausdorff by Theorem \ref{thm5.2}. Moreover, $\gt$ is minimal and effective by Theorems \ref{thm8.3} and \ref{thm9.6}, respectively. Now, proof of the ``if" part of \cite[Theorem 5.1]{bro14} concludes that $C^*_{\mathrm{red}}(\gt)$ is a simple $C^*$-algebra. In particular, in case $\gt$ is amenable, we then have $\og \cong C^*(\gt)=C^*_{\mathrm{red}}(\gt)$, and consequently $\og$ is simple.
\end{proof}

By \cite[Proposition 3.10]{li21-ideal}, whenever $G$ is an amenable group and $\La$ is a row-finite source-free $k$-graph, then $\og$ is nuclear and therefore $\og \cong C^*_{\mathrm{red}}(\gt)$. In this case, the boundary path space $\partial\La$ is just $\La^\infty$. So, we conclude the following generalization of \cite[Theorem 6.6(ii)]{li21} by our inverse semigroup approach.

\begin{cor}\label{cor10.2}
Let $(G,\La)$ be a self-similar $k$-graph over an amenable group $G$ and a row-finite source-free $k$-graph $\La$. Suppose that $\gt$ is Hausdorff (see Theorem \ref{thm5.2}). Then $\og$ is simple if and only if $(G,\La)$ satisfies conditions (1), (2) and (3) in Proposition \ref{prop10.1} above.
\end{cor}

There are several results for minimal and effective properties of groupoids associated to an LCSC (see \cite{spi14, ort20, li23, ort23} for example), which can be of course considered for self-similar $k$-graphs by constructing the Zappa-Sz$\acute{\mathrm{e}}$p product $\La \rtimes^{\varphi} G$ and applying Proposition \ref{prop6.5}. In particular, our Theorems \ref{thm8.3} and \ref{thm9.6} describe \cite[Theorems 6.4, 6.6 and Propositions 7.7, 7.8]{ort20} by giving equivalent self-similar $k$-graphical criteria for such properties, and Corollary \ref{cor10.2} is an analogous version of \cite[Theorems 6.7 and 7.9]{ort20} in the self-similar $k$-graph setting.

\begin{rem}
It seems that in the  the simplicity results \cite[Theorems 6.7 and 7.9]{ort20} for LCSC $C^*$-algebras, the Hausdorffness of groupoid is necessary. For instance, \cite[Corollary 5.26]{cla19} provides a class of non-simple Nekrashevych algebras $A_{\mathbb{Z}_2}(\mathcal{G}_{\mathrm{tight}}(\mathcal{S}_{G,X}))$ such that the underlying groupoids are minimal and effective, but not Hausdorff (see also \cite[Example 4.5]{nek16}). Since such algebras may be presented via self-similar 1-graphs (with a single vertex) \cite{exe17} and for every self-similar 1-graph $(G,E,\varphi)$ we have $\widehat{\mathcal{E}}_\infty(\mathcal{S}_{G,E})=\widehat{\mathcal{E}}_{\mathrm{tight}}(\mathcal{S}_{G,E})=\partial E$ (see the paragraph after Proposition \ref{prop6.4}), it is a counter example for Theorems 6.7 and 7.9 of \cite{ort20} in the case that $\mathcal{G}_{\mathrm{tight}}$ is non-Hausdorff (see also \cite[Section 6.3]{ste21}).
\end{rem}


\section{Simplicity of $\og$: the non-Hausdorff case}\label{sec11}

The simplicity of $C^*_{\mathrm{red}}(\mathcal{G})$ for a non-Hausdorff $\acute{\mathrm{e}}$tale groupoid $\mathcal{G}$ is more complicated (see \cite{cla19}). In particular, for the tight groupoid of an inverse semigroup $S$, one may follow \cite[Subsection 5.2]{cla19} to find certain conditions to insure that $C^*_{\mathrm{red}}(\mathcal{G}_{\mathrm{tight}}(S))$ is a simple $C^*$-algebra. More concretely, it is proposed a stronger condition than the topological freeness for the action $\theta: S\curvearrowright \widehat{\mathcal{E}}_{\mathrm{tight}}(S)$, called Condition (S), such that it together with the irreducibility conclude the simplicity of $C^*_{\mathrm{red}}(\mathcal{G}_{\mathrm{tight}}(S))$. In the final section of this article, we consider again a self-similar $k$-graph $(G,\La)$ over a row-finite source-free $k$-graph $\La$ and examine this condition for the action $\theta: \sg \curvearrowright \widehat{\mathcal{E}}_{\mathrm{tight}}(\sg)$.

Let us first recall Condition (S). Let $S$ be an inverse semigroup with zero. For $s\in S$, we denote by $\mathrm{Fix}(s)$ and $\mathrm{TFix}(s)$ respectively the sets of fixed points and trivially fixed points for $s$, i.e.
$$\mathrm{Fix}(s)=\left\{x\in D^{s^* s}: \theta_s(x)=x \right\}$$
and
$$\mathrm{TFix}(s)=\left\{ x\in \mathrm{Fix}(s): x\in D^e \mathrm{~for~ some~ idempotent~} e\leq s \right\}.$$

\begin{defn}[{\cite[Definition 5.4]{cla19}}]
We say an inverse semigroup $S$ satisfies \emph{Condition (S)} if for any finite set $\{s_1,\ldots, s_l\}\subseteq S\setminus \mathcal{E}(S)$ and any $x\in \bigcap_{i=1}^l \mathrm{Fix}(s_i)$, then
\begin{equation}\label{eq11.1}
x\in \bigcap_{i=1}^l \big(\mathrm{Fix}(s_i)\setminus \mathrm{TFix}(s_i)\big) ~~ \Longrightarrow ~~ x\notin \left(\bigcup_{i=1}^l \mathrm{Fix}(s_i)  \right)^\circ.
\end{equation}
\end{defn}

\begin{lem}\label{lem11.2}
Let $(G,\La)$ be a self-similar $k$-graph over a row-finite, source-free $k$-graph $\La$. Suppose that $(G,\La)$ satisfies the $G$-aperiodic condition (A) (Definition \ref{defn9.2}). If $F_1,\ldots,F_l$ are finitely many singletons in $\sg$ of the form $F_i=\{(\mu_i,g_i,\nu_i)\}$ and $x\in \left(\bigcup_{i=1}^l \mathrm{Fix}(F_i)  \right)^\circ$, then there exists $1\leq j\leq l$ such that $d(\mu_j)=d(\nu_j)$.
\end{lem}

\begin{proof}
Let $x$ be an interior point of $\bigcup_{i=1}^l \mathrm{Fix}(F_i)$. Since the cylinders $Z(x(0,n))$ are a neighbourhood base at $x$, there is $n\in \mathbb{N}^k$ such that $Z(x(0,n))\subseteq \bigcup_{i=1}^l \mathrm{Fix}(F_i)$. So, every $y\in Z(x(0,n))$ is a fixed point for some $F_i$; that is
\begin{align}\label{eq11.2}
\theta_{F_i}(y)&=y \nonumber\\
\Longrightarrow \hspace{17mm} \theta_{\{(\mu_i,g_i,\nu_i)\}}(y)&=y \nonumber \\
\Longrightarrow \hspace{13mm} \mu_i(g_i\cdot\sigma^{d(\nu_i)}(y)) & =y =\mu_i\big(\sigma^{d(\mu_i)}(y) \big) \\
\Longrightarrow \hspace{19.5mm}  g_i\cdot\sigma^{d(\nu_i)}(y)&=\sigma^{d(\mu_i)}(y). \nonumber
\end{align}
If we write $y=x(0,n)z$, where $z\in x(n)\La^\infty$, then Lemma \ref{lem9.4} yields
\begin{equation}\label{eq11.3}
\varphi(g_i,x(0,n)) \cdot \sigma^{d(\nu_i)}(z)=\sigma^{d(\mu_i)}(z).
\end{equation}
It says that for every $z\in x(n)\La^\infty$ there exists an index $1\leq i\leq l$ such that (\ref{eq11.3}) holds.

On the other hand, by the $G$-aperiodicity at $x(n)$, there is $z\in x(n) \La^\infty$ such that $g\cdot \sigma^m(z)\neq \sigma^n(z)$ for all $m\neq n \in \mathbb{N}^k$ and $g\in G$. Hence, if $d(\mu_i)\neq d(\nu_i)$ for all $1\leq i\leq l$, this contradicts (\ref{eq11.3}). Consequently, we must have $d(\mu_j)=d(\nu_j)$ for some $1\leq j\leq l$ as desired.
\end{proof}

\begin{prop}\label{prop11.3}
Let $(G,\La)$ be a self-similar $k$-graph over a row-finite, source-free $k$-graph $\La$. If $(G,\La)$ satisfies the following properties:
\begin{enumerate}[(1)]
  \item $(G,\La)$ is $G$-cofinal (Definition \ref{defn8.2}),
  \item $(G,\La)$ satisfies the $G$-aperiodic property (A), and
  \item If $g\cdot x=x$ for some $g\in G$ and $x\in \La^\infty$, then there is $n\in \mathbb{N}^k$ such that $\varphi(g,x(0,n))=e_G$,
\end{enumerate}
then $C_{\mathrm{red}}^*(\gt)$ is a simple $C^*$-algebra.
\end{prop}

\begin{proof}
We first show that conditions (2) and (3) above deduce Condition (S) for the action $\theta:\sg \curvearrowright \widehat{\mathcal{E}}_{\mathrm{tight}}(\sg)$. To do this, fix $F_1,\ldots,F_l\in \sg$ and some $x\in \bigcap_{i=1}^l \mathrm{Fix}(F_i)$. By Definition \ref{defn3.1}, each $F_i$ is a finite set of orthogonal triples $(\mu,g,\nu)$, so there is a unique triple $(\mu_i,g_i,\nu_i)$ in $F_i$ such that $x\in \mathrm{Fix}(\{(\mu_i,g_i,\nu_i)\})$. Thus, without loss of generality, we may simply assume that every $F_i$ is a singleton set of the form $F_i=\{(\mu_i,g_i,\nu_i)\}$.

We prove Condition (S) by the contrapositive of (\ref{eq11.1}). If $x\in \left(\bigcup_{i=1}^l \mathrm{Fix}(F_i)  \right)^\circ$, Lemma \ref{lem11.2} says that there is an index $1\leq j\leq l$ such that $d(\mu_j)=d(\nu_j)$. As $x\in \mathrm{Fix}(F_j)$, $x$ must be of the form $x=\nu_j y$ for $y\in \La^\infty$ such that
\begin{align*}
\theta_{F_j}(x)&=x \\
\Longrightarrow \hspace{18mm} \theta_{\{(\mu_j,g_j,\nu_j)\}}(\nu_j y)&=\nu_j y \hspace{20mm} \\
\Longrightarrow \hspace{29mm} \mu_j(g_j\cdot y) & =\nu_j y.
\end{align*}
Since $d(\mu_j)=d(\nu_j)$, the unique factorization property implies $\mu_j=\nu_j$ and $g_j\cdot y = y$. Next, by condition (3) there exists $n\in \mathbb{N}^k$ such that  $\varphi(g_j,y(0,n))=e_G$. Write $\la:=y(0,n)$ for convenience. Then $\la$ is strongly fixed by $g_j$, and Lemma \ref{lem4.3} yields that $F_j$ dominates the idempotent $\{(\mu_j\la, e_G, \mu_j\la)\}$. Hence, we have $x\in \mathrm{TFix}(F_j)$ contradicting the left hand side of (\ref{eq11.1}). Consequently, the inverse semigroup $\sg$ satisfies Condition (S).

Now, since $\widehat{\mathcal{E}}_{\mathrm{tight}}(\sg)=\widehat{\mathcal{E}}_\infty (\sg)$, \cite[Lemma 5.6]{cla19} implies that every compact open subset of $\gt$ is regular open\footnote{A subset $B\subseteq \gt$ is called \emph{regular open} if $(\overline{B})^0=B$}. Moreover, $\gt$ is a minimal and effective groupoid by Theorems \ref{thm8.3} and \ref{thm9.6}, respectively. Therefore, we may apply \cite[Corollary 4.12]{cla19} to deduce that $C^*_{\mathrm{red}}(\gt)$ is a simple $C^*$-algebra.
\end{proof}

\begin{cor}\label{cor11.4}
Let $(G,\La)$ be a self-similar $k$-graph over a row-finite, source-free $k$-graph $\La$ and an ``amenable" group $G$. If $(G,\La)$ satisfies conditions (1)-(3) of Propositions \ref{prop11.3} above, then $\og$ is simple.
\end{cor}

\begin{proof}
Since $G$ is amenable, \cite[Proposition 3.10]{li21-ideal} implies that $\og$ is a nuclear $C^*$-algebra. Hence, $\og \cong C^*_{\mathrm{red}}(\gt)$ and the result follows from Proposition \ref{prop11.3}.
\end{proof}


\subsection*{Acknowledgment}

The author is grateful to Dilian Yang and the anonymous referee for reviewing the initial version of article and their constructive comments. Also, he acknowledges financial support from Shahid Chamran University of Ahvaz (grant no. SCU.MM1402.279).


\subsection*{Conflict of interest}
The author declares that he has no conflict of interest.


\end{document}